\documentclass[10pt]{article}
\usepackage{amsmath,amssymb}
\usepackage{enumerate}

\allowdisplaybreaks[1]

\newtheorem{theorem}{Theorem}[section]
\newtheorem{lemma}[theorem]{Lemma}
\newtheorem{proposition}[theorem]{Proposition}

\newtheorem{definition}[theorem]{Definition}
\newtheorem{remark}[theorem]{Remark}
\newtheorem{example}[theorem]{Example}

\newtheorem{problem}[theorem]{Problem}

\newcommand{\Z}{\mathbb{Z}}

\newcommand{\im}{\operatorname{Im}}
\renewcommand{\ker}{\operatorname{Ker}}
\newcommand{\id}{\operatorname{id}}
\newcommand{\Id}{\operatorname{Id}}
\newcommand{\sym}{\operatorname{Sym}}
\newcommand{\aut}{\operatorname{Aut}}
\newcommand{\End}{\operatorname{End}}

\newcommand{\soc}{\operatorname{Soc}}
\newcommand{\GL}{\operatorname{GL}}
\newcommand{\Aut}{\operatorname{Aut}}

\newenvironment{proof}{\par\noindent{\bf Proof.}}{$\qed$\par\bigskip}
\newcommand{\qed}{\enspace\vrule  height6pt  width4pt  depth2pt}
\usepackage{color}

\begin{document}

\title{Iterated matched products of finite braces and simplicity; new solutions of the Yang-Baxter equation\thanks{The two first-named authors were partially supported by the grants
DGI MICIIN MTM2011-28992-C02-01, and MINECO MTM2014-53644-P. The
third author is supported in part by Onderzoeksraad of Vrije
Universiteit Brussel and Fonds voor Wetenschappelijk Onderzoek
(Belgium). The fourth author is supported by the National Science
Centre grant 2013/09/B/ST1/04408 (Poland).
\newline {\em 2010 MSC:} Primary 16T25, 20E22, 20F16.
\newline {\em Keywords:} Yang-Baxter equation, set-theoretic solution,
brace, simple, matched product.}}
\author{D. Bachiller \and F. Ced\'o \and E. Jespers \and J. Okni\'{n}ski}
\date{}

\maketitle

\begin{abstract}
Braces were introduced by Rump as a promising tool in the study of
the set-theoretic solutions of the Yang-Baxter equation. It has
been recently proved that, given a left brace $B$, one can
construct explicitly all the non-degenerate involutive
set-theoretic solutions of the Yang-Baxter equation such that the
associated permutation group is isomorphic, as a left brace, to
$B$. It is hence of fundamental importance to describe all
simple objects in the class of finite left braces. In this paper
we focus on the matched product decompositions of an arbitrary
finite left brace. This is used to construct new families of
finite simple left braces.
\end{abstract}

\section{Introduction}
Braces were introduced by Rump \cite{Rump} to study a class of
solutions of the Yang-Baxter equation, a fundamental equation in
mathematical physics that has become, since its origin in a paper of
Yang \cite{Yang}, a key ingredient in quantum groups and Hopf
algebras \cite{K}. The primary aim of this paper is to present new
general constructions of finite braces, with the main focus on
constructing finite simple braces. The latter is the key step in the
challenging problem of a classification of finite simple braces. Our
approach is based on the notion of iterated matched product of
braces, which turns out to be an indispensable tool in this context.

Recall that a solution of the quantum Yang-Baxter equation is a
linear map $\mathcal{R}\colon V\otimes V\longrightarrow V\otimes V$,
for a vector space $V$, such that
\begin{eqnarray}\label{qYBeq}
&&\mathcal{R}_{12}\mathcal{R}_{13}\mathcal{R}_{23}=\mathcal{R}_{23}\mathcal{R}_{13}\mathcal{R}_{12},
\end{eqnarray}
where $\mathcal{R}_{ij}$ denotes the map $V\otimes V \otimes V
\longrightarrow V\otimes V \otimes V$ acting as $\mathcal{R}$ on the
$(i,j)$ tensor factor and as the identity on the remaining factor. A central and difficult open problem is
to construct new families of solutions of this equation. An
equivalent problem is to find solutions of the Yang-Baxter equation
\begin{eqnarray}\label{YBeq}
&&R_{12}R_{23}R_{12}=R_{23}R_{12}R_{23}.
\end{eqnarray}
In fact, if $\tau\colon V\otimes V\longrightarrow V\otimes V$ is the
linear map such that $\tau (u\otimes v)=v\otimes u$ for all $u,v\in
V$, then it is easy to check that $\mathcal{R}\colon V\otimes V
\longrightarrow V\otimes V$ is a solution of the quantum Yang-Baxter
equation~(\ref{qYBeq}) if and only if $R=\tau\circ \mathcal{R}$ is a
solution of the Yang-Baxter equation~(\ref{YBeq}). In the context of
quantum groups and Hopf algebras  such solutions are often referred
to as R-matrices (see for example \cite{brown,K}).  Drinfeld
 in \cite{drinfeld} initiated the investigations  of the set-theoretic solutions of the
Yang-Baxter equation, these are the maps $r\colon X\times
X\longrightarrow X\times X$ such that
\begin{eqnarray}\label{stYBeq}
&&r_{12}r_{23}r_{12}=r_{23}r_{12}r_{23},
\end{eqnarray}
where $r_{ij}$ denotes the map $X\times X \times X \longrightarrow
X\times X \times X$ acting as $r$ on the $(i,j)$ components and as
the identity on the remaining component. Note that if $X$  is  a
basis of the vector space $V$, then every such solution $r$ induces
a solution $R\colon V\otimes V\longrightarrow V\otimes V$ of the
Yang-Baxter equation~(\ref{YBeq}).

Gateva-Ivanova and Van den Bergh \cite{GIVdB}, and Etingof, Schedler
and Soloviev \cite{ESS} introduced a subclass of the set-theoretic
of solutions, the non-degenerate involutive solutions. Recall that a
set-theoretic solution $r\colon X\times X\longrightarrow X\times X$
of the Yang-Baxter equation~(\ref{stYBeq}),  written in the form
$r(x,y)=(f_x(y),g_y(x))$ for  $x,y\in X$, is involutive if
$r^{2}=\id_{X^{2}}$, and it is non-degenerate if $f_x$ and $g_x$ are
bijective maps from $X$ to $X$, for all $x\in X$. This class of
solutions has received a lot of attention in  recent years, see for
example  \cite{CCS2015,CCS,CJO, CJO2, CJR, ESS, GI, GIC, gat-maj,
GIVdB, JO, JObook, Rump1, Rump,ven2016}. Braces were introduced to
study this  type  of solutions. Recall that a left brace is a set
$B$ with two operations, $+$ and $\cdot$,
 such that $(B,+)$ is an abelian group, $(B,\cdot)$ is a group and
\begin{eqnarray}\label{lbrace}
&& a\cdot (b+c)+a=a\cdot b+a\cdot c,
\end{eqnarray}
for all $a,b,c\in B$. A  right brace is defined similarly, but
replacing property (\ref{lbrace}) by $(b+c)\cdot a+a=b\cdot a+
c\cdot a$. If $B$ is both a left and a right brace (for the same
operations), then one says that $B$ is  a two-sided brace. Rump
initiated the study of this new algebraic structure, though, using
another but equivalent definition  \cite{Rump3, Rump, Rump2, Rump4,
Rump5, Rump6, Rump7}. In particular, he noted that the structure
group $G(X,r)$ of a non-degenerate, involutive set-theoretic
solution $(X,r)$ of the Yang-Baxter equation (solution of the YBE
for short) admits a natural structure of left brace, such that its
additive group is the free abelian group with basis $X$ and
$xy-x=f_x(y)$ for all $x,y\in X$. The structure group $G(X,r)$ was
introduced and studied in \cite{ESS, GIVdB} and it is defined as the
group with the following presentation
$$G(X,r)= \langle X\mid xy=f_x(y)g_y(x)\mbox{ for all }x,y\in X\rangle.$$

Another important group associated to a solution $(X,r)$ of the YBE
is its permutation group $\mathcal{G}(X,r)$, which is the subgroup
of the symmetric group $\sym_X$ on $X$ generated by $\{ f_x\mid x\in
X\}$.  The map $x\mapsto f_x$ extends to a group  epimorphism
$\phi\colon G(X,r)\longrightarrow \mathcal{G}(X,r)$  with kernel
$\ker(\phi)=\{ a\in G(X,r)\mid ab=a+b$ for all $b\in G(X,r) \}$. The
group $\mathcal{G}(X,r)$ inherits a natural left brace structure so
that $\phi$ becomes a homomorphism of left braces.

Some important open problems  have been solved in
\cite{CJO2} using braces.
 Several aspects of the theory of braces and their applications in the context of the
 Yang-Baxter equation have been recently considered  in \cite{GI-braces, Smokt} and
 \cite{CGIS}.
 It is known that every finite left brace is isomorphic to
$\mathcal{G}(X,r)$ (as left braces) for some finite solution $(X,r)$
of the YBE (\cite[Theorem 2]{CJO2}). Thus, by \cite[Theorem~2.15]{ESS}, the multiplicative
group of every finite left brace is solvable. But not all  finite
solvable groups are isomorphic to the multiplicative group of a left
brace \cite{B2}. In fact, there exist finite $p$-groups that are not
isomorphic to the multiplicative group of any left brace.  Given a
left brace $B$, in \cite{BCJ}  a method is  given to construct
explicitly  all the solutions $(X,r)$ of the YBE such that
$\mathcal{G}(X,r)\cong B$ as left braces. Therefore, the problem of
constructing all the solutions of the YBE is reduced to describing
all left braces. The challenging problem of classifying all finite
left braces naturally splits into two parts:
\begin{itemize}
\item[(a)] Classify the simple objects in the class of   finite left braces.
\item[(b)] Develop an appropriate   theory of extensions of left braces.
\end{itemize}
Note that, by  \cite[Corollary~II.6.12]{Cohn},  a version of
Jordan-H\"{o}lder theorem holds for finite left braces; emphasizing
the importance of simple  left braces.

Recall that an ideal of a left brace $B$ is a normal subgroup $I$ of
its multiplicative group such that $\lambda_b(a)\in I$ for all $a\in
I$ and $b\in B$, where $\lambda_b$ is the automorphism of $(B,+)$
defined by $\lambda_b(c)=bc-b$, for all $b,c\in B$. For example, the
socle $\soc (B) =\{ b\in B \mid \lambda_b =\id \}$ is an ideal of
$B$. A left brace $B$ is  said to be  simple if it is nonzero and
$\{0\}$ and $B$ are the only ideals of $B$. Recall that a left brace
is  said to be  trivial if its multiplication coincides with its
addition. The socle of an arbitrary  left brace $B$ is a trivial
brace. It is known that every simple left brace of prime power order
$p^n$ is a trivial brace of cardinality $p$,  \cite[Corollary on page 166]{Rump}. Until
recently, these were the only known examples of finite simple left
braces. The first finite nontrivial simple left braces have been
constructed in \cite[Theorem~6.3 and Section 7]{B3};
 the additive groups of which
are isomorphic  to $\mathbb{Z}/(p_1)\times
(\mathbb{Z}/(p_2))^{k(p_1-1)+1}$, where $p_1,p_2$ are primes such
that $p_2\mid p_1-1$  and $k$ is a positive integer. We shall give
a much larger class of examples based on the construction of
matched products of braces, which is introduced in \cite{B3} as a
natural extension of the matched product (or bicrossed product) of
groups \cite{K}.  Note that matched products of groups also
appear in the context of solutions of the YBE, see for example the
survey of Takeuchi \cite{Tak} and \cite{GI-braces}.

  Every left brace $B$ admits a left action
$\lambda\colon (B,\cdot)\longrightarrow \Aut(B,+)$ defined by
$\lambda(b)=\lambda_b$ for all $b\in B$
(see \cite[Lemma~1]{CJO2}). It is called the lambda map of the left  brace $B$.
Recall that, given the lambda map of a left brace $B$,  each of the structures $(B,\cdot)$ and $(B,+)$ determines the other one uniquely.

Lambda maps are used to define the matched products of left braces.

\begin{definition}\label{matchedpair}
Let $G$ and $H$ be two left braces. Let $\alpha:
(H,\cdot)\longrightarrow \aut(G,+)$ and $\beta: (G,\cdot)\longrightarrow \aut(H,+)$ be group
homomorphisms. One says that $(G,H,\alpha,\beta)$ is a matched pair of
left braces if the following conditions hold:
\begin{itemize}
\item[$(\mathrm{MP}1)$] $\lambda_a^{(1)}\circ\alpha_b=\alpha_{\beta_a(b)}\circ\lambda^{(1)}_{\alpha^{-1}_{\beta_a(b)}(a)}$,
\item[$(\mathrm{MP}2)$]
$\lambda_b^{(2)}\circ\beta_a=\beta_{\alpha_b(a)}\circ\lambda^{(2)}_{\beta^{-1}_{\alpha_b(a)}(b)}$,
\end{itemize}
where $\alpha(b)=\alpha_b$ and $\beta(a)=\beta_a$, for all $a\in G$
and $b\in H$, {\ with $\lambda^{(1)}$ and $\lambda^{(2)}$ denoting} the lambda
maps of the left braces $G$ and $H$, respectively.
\end{definition}

Let $(G,H,\alpha,\beta)$ be a matched pair of left braces. Then by
\cite[Theorem~4.2]{B3} $G\times H$ is a left brace with addition
$$
(a,b)+(a',b')  = (a+a',b+b'),
$$
and with lambda map given by
$$
\lambda_{(a,b)}(a',b')=\left(\alpha_b\lambda^{(1)}_{\alpha^{-1}_b(a)}(a'),~\beta_a\lambda^{(2)}_{\beta^{-1}_a(b)}(b')\right).
$$

\begin{definition}\label{matchedproduct}
Let $(G,H,\alpha,\beta)$ be a matched pair of left braces. The
left brace defined as above is called the matched product of $G$ and
$H$. We simply denote it by $G\bowtie H$.
\end{definition}

Note that, if $\beta $ is trivial, then we get a semidirect product $G\rtimes H$ of left braces, considered
in \cite{Rump5} and \cite[Section 6]{CJO2}, and then $G$ is an ideal of $G\bowtie H$,
and if additionally $\alpha $ is trivial then we get the direct product $G\times H$ of left braces.

Of course, if $(G,H,\alpha,\beta)$ is a matched pair of left braces
then so is $(H,G,\beta, \alpha)$. Furthermore, it easily is verified
that the map $G\bowtie H \longrightarrow H\bowtie G$ defined by $(a,b)
\mapsto (b,a)$ is an isomorphism of left braces.

Recall that a left ideal of a left brace $B$ is a subgroup $S$ of
its multiplicative group such that $\lambda_b(a)\in S$ for all $a\in
S$ and all $b\in B$. Note that for every $b,a \in S$ we have $b-a =
\lambda_{a}(a^{-1}b)$. Thus, in particular, $S$ is a left subbrace
of $B$. If $G\bowtie H$ is a matched product of left braces, then
$G\times \{0\}$ and $\{0\}\times H$ are left ideals of $G\bowtie H$.
Conversely, it is not difficult to see (use for example Lemma 2 in
\cite{CJO2} to verify conditions (MP1) and (MP2)) that if $B$ is a
left brace and $B_1,B_2$ are two left ideals of $B$ such that
$(B,+)$ is the inner direct sum of $(B_1,+)$ and $(B_2,+)$, then
$(B_1,B_2,\alpha,\beta)$ is a matched pair of left braces, where
$\alpha_b(a)=ba-b$ and $\beta_a(b)=ab-a$, for all $a\in B_1$ and
$b\in B_2$. Furthermore, the map $\eta: B_1\bowtie
B_2\longrightarrow B$ defined by $\eta(a,b)=a+b$, for all $a\in B_1$
and $b\in B_2$, is an isomorphism of left braces.

Our main starting point is the following observation, contained
implicitly in \cite[Section 4]{B3}.

\begin{remark}\label{Sylow_remark}
{\rm Let $B$ be a finite nonzero left brace. Then there exist
distinct prime numbers $p_1,\ldots, p_k$ and left braces
$H_1,\ldots, H_k$, with $k\geq 1$,  such that $|H_i|=p_{i}^{n_i}$
and $B$ is an iterated matched product $B=(\dots (H_1\bowtie
H_2)\bowtie \dots) \bowtie H_k$ of  left braces.  Moreover, each
$H_i$ is a left ideal of $B$.}
\end{remark}

Essentially, the key argument used in the proof is that the Sylow
subgroups of $(B,+)$ are left ideals of $B$. In
Section~\ref{iterated} we give a proof of a more general result,
Theorem~\ref{itlb}.

 Remark~\ref{Sylow_remark} explains why one can construct finite
left braces using matched products,  with braces of prime power
order as the building blocks. In particular, all finite simple
left braces can be constructed in this way.  It is the aim of this
paper to construct a large class of simple braces via this method.
So we focus on part (a) of the classification problem. Some
partial results on the classification of ``small'' left braces can
be found in \cite{B,Ven1,Rump2}. Concerning part (b), i.e.
developing a  theory of  extensions of left braces, some
preliminary results can be found in \cite{B3,BenDavid, BDG2,
Vend-Lebed}.  But a general theory is missing.

In Section~\ref{iterated}, we study iterated matched products of
left braces corresponding to an inner direct sum of left ideals.
First we characterize the iterated matched products of left braces that are of this form. Next,
we give necessary and sufficient conditions for such a matched product to be simple provided the
left ideals are simple left braces.

In Section~\ref{simple}, we first  generalize an intriguing
construction of Heged\H{u}s \cite{H}, developed in the context of
regular subgroups of the affine group, that has been recently
considered also in \cite{B3} and \cite{CCS}.  Next, within this class,  we  construct
iterated matched products  of left braces and we give
necessary and sufficient conditions on the actions corresponding to
these matched products for their simplicity.

In Section~\ref{examples},  we  construct concrete examples of
simple left braces  of the type described in Section~\ref{simple}.
We also show how to construct more examples of finite simple left
braces using the results of Section~\ref{iterated} and thus
indicating that this may be a very rich area to explore.  In
this context, as mentioned before,  it is shown in \cite[Theorem 3.1]{BCJ} how
to describe for a  given left brace $B$,  all solutions of the YBE
with associated permutation group isomorphic to $B$ (as a left
brace). So the new examples constructed in Section~\ref{examples}
provide  new families of solutions of the YBE.

Finally, in Section~\ref{comments}, we state two problems that are
fundamental  for the  program of the  classification of finite
(simple) left braces and thus a description of all finite solutions
of  the YBE. The first problem is concerned with the automorphism
group of a finite left brace of prime power order. The second
problem deals with simplicity of left braces of orders of the form
$p^{n}q^{m}$, for two primes $p,q$. In this context, Smoktunowicz in
\cite{Smokt} recently proved the following property.
\begin{center}{\it
If $|B|=q^{r} k$ with $q$ prime,  $(q, k)=1$, $ k\neq 1$ and $B$ is a
simple left brace  then there exists a prime $p$ such that $q|(p^{t}-1)$ for
some $1\leq t$ and $p^{t}| k$.}
\end{center}
So, if $B$ is a simple left brace of order $p^nq^m$ (with $p$ and
$q$ different prime numbers and $n,m$ positive integers) then $p |
(q^{t}-1)$ and $q| (p^{s}-1)$ for some $0<t\leq m$ and $0<s \leq n$.
We observe that these conditions are not sufficient for simplicity of $B$.

\section{Finite braces as  iterated matched products of left ideals and
simplicity}\label{iterated}

In the first part of this section, motivated by Remark
~\ref{Sylow_remark},  we characterize left braces that are iterated
matched products of subbraces that are left ideals. In the second
part, we prove a simplicity result that is later used to construct
several new families of simple left braces.

Let $B$ be a left brace. Suppose that there exist left ideals
$B_1,\dots ,B_n$ of $B$ such that $n\geq 2$ and the additive group
of $B$ is the direct sum of the additive groups of the left ideals
$B_i$. Denote by $\lambda^{(i)}$ the lambda map corresponding to
$B_i$, which is the restriction to $B_i$ of the lambda map $\lambda$
of $B$. For $1\leq j<n$, denote by $\lambda^{(1,\dots ,j)}$ the lambda map corresponding to $B_1+\dots +B_j$, which is the restriction to
$B_1+\dots +B_j$ of the lambda map of $B$. Since $B_1+\dots +B_j$ and $B_{j+1}$ are left ideals of $B$, the maps
$$\alpha^{((1,\dots ,j),j+1)}\colon (B_1+\dots +B_j,\cdot)\longrightarrow \Aut(B_{j+1},+),\quad\mbox{and}$$
$$\alpha^{(j+1,(1,\dots ,j))}\colon (B_{j+1},\cdot)\longrightarrow \Aut(B_1+\dots+B_j,+)$$
defined by $\alpha^{((1,\dots ,j),j+1)}(x)=\alpha^{((1,\dots
,j),j+1)}_x ,\;$ $\alpha^{(j+1,(1,\dots ,j))}(y)=\alpha^{(j+1,(1,\dots
,j))}_y$ and $\alpha^{((1,\dots ,j),j+1)}_x(y)=xy-x$ and
$\alpha^{(j+1,(1,\dots ,j))}_y(x)=yx-y$, for all $x\in B_1+\dots
+B_j$ and $y\in B_{j+1}$, are homomorphisms of groups. Since
$\lambda_a\circ \lambda_{\lambda_a^{-1}(b)}=\lambda_b\circ
\lambda_{\lambda_b^{-1}(a)}$ (see \cite[Lemma~2]{CJO2}), we have
that
\begin{itemize}
\item[(i)] $\lambda^{(j+1)}_y\circ\alpha^{((1,\dots,j),j+1)}_{(\alpha^{(j+1,(1,\dots,j))}_y)^{-1}(x)}=
\alpha^{((1,\dots,j),j+1)}_x\circ\lambda^{(j+1)}_{(\alpha^{((1,\dots,j),j+1)}_x)^{-1}(y)}$ and
\item[(ii)] $\lambda^{(1,\dots,j+1)}_x\circ\alpha^{(j+1,(1,\dots,j))}_{(\alpha^{((1,\dots,j),j+1)}_x)^{-1}(y)}=
\alpha^{(j+1,(1,\dots,j))}_y\circ\lambda^{(1,\dots,j+1)}_{(\alpha^{(j+1,(1,\dots,j))}_y)^{-1}(x)}$.
\end{itemize}
Therefore $(B_1+\dots+B_j, B_{j+1},\alpha^{(j+1,(1,\dots,j))},\alpha^{((1,\dots,j),j+1)})$ is a matched pair of left braces and, by \cite[Theorem~4.2]{B3}, the map
$$\eta_{1,\dots,j+1}\colon (B_1+\dots+B_j)\bowtie B_{j+1}\longrightarrow B_1+\dots+B_{j+1},$$
defined by $\eta_{1,\dots,j+1}(x,y)=x+y$, is an isomorphism of left braces.

We define the maps $\beta^{(2,(1))}=\alpha^{(2,(1))},\;$ $\beta^{((1),2)}=\alpha^{((1),2)}$ and, for $1< j<n$,
$$\beta^{((1,\dots, j),j+1)}\colon B_1\times \dots\times B_j\longrightarrow \Aut(B_{j+1},+)\quad\mbox{and}$$
$$\beta^{(j+1,(1,\dots, j))}\colon B_{j+1}\longrightarrow \Aut(B_1\times\dots\times B_j,+)$$
by $\beta^{((1,\dots,j),j+1)}(a_1,\dots,a_j)=\beta^{((1,\dots,j),j+1)}_{(a_1,\dots,a_j)},\;$ $\beta^{(j+1,(1,\dots,j))}(a_{j+1})
=\beta^{(j+1,(1,\dots,j))}_{a_{j+1}}$,
$$\beta^{((1,\dots,j),j+1)}_{(a_1,\dots,a_j)}(a_{j+1})=\alpha^{((1,\dots,j),j+1)}_{a_1+\dots+a_j}(a_{j+1})
\quad\mbox{and}$$
 $$\beta^{(j+1,(1,\dots,j))}_{a_{j+1}}(a_1,\dots,a_j)=\eta_j^{-1}\alpha^{(j+1,(1,\dots,j))}_{a_{j+1}}\eta_j(a_1,\dots,a_j),$$ where
$\eta_k\colon B_1\times \dots\times B_k\longrightarrow B_1+\dots+B_k$ is defined by
$\eta_k(a_1,\dots,a_k)=a_1+\dots+a_k$, for $1<k\leq n$.

\begin{proposition}\label{inner}
With the above notation,
$$((\dots(B_1\bowtie B_2)\bowtie\dots)\bowtie B_j,B_{j+1},\beta^{(j+1,(1,\dots,j))},\beta^{((1,\dots,j),j+1)})$$
is a matched pair of left braces and the map  $\eta_{j+1}\colon
(\dots(B_1\bowtie B_2)\bowtie\dots)\bowtie B_{j+1}\longrightarrow
B_1+\dots+B_{j+1}$ is an isomorphism of left braces, for all $1\leq
j<n$.
\end{proposition}

\begin{proof} We will prove the result by induction on $j$. For $j=1$, the result follows
because $\beta^{(2,(1))}=\alpha^{(2,(1))}$ and $\beta^{((1),2)}=\alpha^{((1),2)}$.
Suppose that $j>1$ and the result is true for $j-1$. By the induction hypothesis, $\eta_{j}$
is an isomorphism of left braces. Thus, since $\alpha^{((1,\dots,j),j+1)}$ and $\alpha^{(j+1,(1,\dots,j))}$ are homomorphisms of groups, we have that
$\beta^{((1,\dots,j),j+1)}$ and $\beta^{(j+1,(1,\dots,j))}$ are homomorphisms of groups.
Let $\tilde\lambda^{(1,\dots,j)}$ be the lambda map of the left brace $(\dots(B_1\bowtie B_2)\bowtie\dots)\bowtie B_j$ corresponding to the matched pair of left braces
$((\dots(B_1\bowtie B_2)\bowtie\dots)\bowtie B_{j-1},B_{j},\beta^{(j,(1,\dots,j-1))},\beta^{((1,\dots,j-1),j)})$.
Since $\eta_j$ is an isomorphism of left braces,
$$\tilde\lambda^{(1,\dots, j)}_{(a_1,\dots ,a_j)}=\eta_j^{-1}\lambda^{(1,\dots, j)}_{a_1+\dots +a_j}\eta_j.$$
 To show that $((\dots(B_1\bowtie
B_2)\bowtie\dots)\bowtie
B_j,B_{j+1},\beta^{(j+1,(1,\dots,j))},\beta^{((1,\dots,j),j+1)})$ is
a matched pair of left braces,  we should check the following
equalities.
\begin{itemize}
\item[(i)] $\tilde\lambda^{(1,\dots, j)}_{(a_1,\dots ,a_j)}\circ \beta^{(j+1,(1,\dots,j))}_{(\beta^{((1,\dots,j),j+1)}_{(a_1,\dots,a_j)})^{-1}(a_{j+1})}
=\beta^{(j+1,(1,\dots,j))}_{a_{j+1}}\circ\tilde\lambda^{(1,\dots, j)}_{(\beta^{(j+1,(1,\dots,j))}_{a_{j+1}})^{-1}(a_1,\dots,a_j)}$,
\item[(ii)] $\lambda^{(j+1)}_{a_{j+1}}\circ \beta^{((1,\dots,j),j+1)}_{(\beta^{(j+1,(1,\dots,j))}_{a_{j+1}})^{-1}(a_1,\dots,a_j)}
=\beta^{((1,\dots,j),j+1)}_{(a_1,\dots,a_j)}\circ\lambda^{(j+1)}_{(\beta^{((1,\dots,j),j+1)}_{(a_1,\dots,a_j)})^{-1}(a_{j+1})}$.
\end{itemize}
 We prove first equality (i).
\begin{eqnarray*}
\lefteqn{\tilde\lambda^{(1,\dots, j)}_{(a_1,\dots ,a_j)}\circ \beta^{(j+1,(1,\dots,j))}_{(\beta^{((1,\dots,j),j+1)}_{(a_1,\dots,a_j)})^{-1}(a_{j+1})}}\\
&=&\eta_j^{-1}\circ\lambda^{(1,\dots, j)}_{a_1+\dots +a_j}\circ \alpha^{(j+1,(1,\dots,j))}_{(\alpha^{((1,\dots,j),j+1)}_{a_1+\dots+a_j})^{-1}(a_{j+1})}\circ\eta_j\\
&=&\eta_j^{-1}\circ\alpha^{(j+1,(1,\dots,j))}_{a_{j+1}}\circ\lambda^{(1,\dots, j)}_{(\alpha^{(j+1,(1,\dots,j))}_{a_{j+1}})^{-1}(a_1+\dots+a_j)}\circ\eta_j\\
&=&\beta^{(j+1,(1,\dots,j))}_{a_{j+1}}\circ\tilde\lambda^{(1,\dots, j)}_{(\beta^{(j+1,(1,\dots,j))}_{a_{j+1}})^{-1}(a_1,\dots,a_j)}.
\end{eqnarray*}
Now we prove equality (ii).
\begin{eqnarray*}
\lefteqn{\lambda^{(j+1)}_{a_{j+1}}\circ \beta^{((1,\dots,j),j+1)}_{(\beta^{(j+1,(1,\dots,j))}_{a_{j+1}})^{-1}(a_1,\dots,a_j)}}\\
&=&\lambda^{(j+1)}_{a_{j+1}}\circ \alpha^{((1,\dots,j),j+1)}_{(\alpha^{(j+1,(1,\dots,j))}_{a_{j+1}})^{-1}(a_1+\dots+a_j)}\\
&=&\alpha^{((1,\dots,j),j+1)}_{a_1+\dots+a_j}\circ\lambda^{(j+1)}_{(\alpha^{((1,\dots,j),j+1)}_{a_1+\dots+a_j})^{-1}(a_{j+1})}\\
&=&\beta^{((1,\dots,j),j+1)}_{(a_1,\dots,a_j)}\circ\lambda^{(j+1)}_{(\beta^{((1,\dots,j),j+1)}_{(a_1,\dots,a_j)})^{-1}(a_{j+1})}.
\end{eqnarray*}
 Hence
$((\dots(B_1\bowtie B_2)\bowtie\dots)\bowtie
B_j,B_{j+1},\beta^{(j+1,(1,\dots,j))},\beta^{((1,\dots,j),j+1)})$ is
a matched pair of left braces. Now
$\eta_{j+1}=\eta_{1,\dots,j+1}\circ(\eta_j\times \id)$. Since
clearly $\eta_{j+1}$ is an isomorphism of the additive groups and
$\eta_{1,\dots,j+1}$ is an isomorphism of left braces, it is enough
to prove that
$$(\eta_j\times \id)(\tilde\lambda^{(1,\dots,j+1)}_{(a_1,\dots,a_{j+1})}(b_1,\dots,b_{j+1}))=
\lambda^{((1,\dots,j),j+1)}_{(a_1+\dots+a_j,a_{j+1})}(b_1+\dots+b_j,b_{j+1}),$$
 where $\lambda^{((1,\dots ,j),j+1)}$ is the lambda map of the
left brace $(B_1+\dots +B_j)\bowtie B_{j+1}$. We have that
\begin{eqnarray*}
\lefteqn{(\eta_j\times \id)(\tilde\lambda^{(1,\dots,j+1)}_{(a_1,\dots,a_{j+1})}(b_1,\dots,b_{j+1}))}\\
&=&(\eta_j\times \id)(\beta^{(j+1,(1,\dots,j))}_{a_{j+1}}\tilde\lambda^{(1,\dots,j)}_{(\beta^{(j+1,(1,\dots,j))}_{a_{j+1}})^{-1}(a_1,\dots,a_{j})}(b_1,\dots,b_{j}),\\
&&\qquad\beta^{((1,\dots,j),j+1)}_{(a_1,\dots,a_j)}\lambda^{(j+1)}_{(\beta^{((1,\dots,j),j+1)}_{(a_1,\dots,a_j)})^{-1}(a_{j+1})}(b_{j+1}))\\
&=&(\eta_j\times \id)(\eta_j^{-1}\alpha^{(j+1,(1,\dots,j))}_{a_{j+1}}\lambda^{(1,\dots,j)}_{(\alpha^{(j+1,(1,\dots,j))}_{a_{j+1}})^{-1}(a_1+\dots+a_{j})}\eta_j(b_1,\dots,b_{j}),\\
&&\qquad\alpha^{((1,\dots,j),j+1)}_{(a_1+\dots+a_j)}\lambda^{(j+1)}_{(\alpha^{((1,\dots,j),j+1)}_{(a_1+\dots+a_j)})^{-1}(a_{j+1})}(b_{j+1}))\\
&=&(\alpha^{(j+1,(1,\dots,j))}_{a_{j+1}}\lambda^{(1,\dots,j)}_{(\alpha^{(j+1,(1,\dots,j))}_{a_{j+1}})^{-1}(a_1+\dots+a_{j})}(b_1+\dots+b_{j}),\\
&&\qquad\alpha^{((1,\dots,j),j+1)}_{(a_1+\dots+a_j)}\lambda^{(j+1)}_{(\alpha^{((1,\dots,j),j+1)}_{(a_1+\dots+a_j)})^{-1}(a_{j+1})}(b_{j+1}))\\
&=&\lambda^{((1,\dots,j),j+1)}_{(a_1+\dots+a_j,a_{j+1})}(b_1+\dots+b_j,b_{j+1}).\end{eqnarray*}
Therefore the result follows.
\end{proof}

Note that with the above notation
$$\{ 0\}\times\dots\times\{ 0\}\times B_i\times\{ 0\}\times\dots\times\{ 0\}$$
is a left ideal of $(\dots(B_1\bowtie B_2)\bowtie\dots)\bowtie B_n$.
This motivates the following definition.

\begin{definition}
Let $B_1,\dots,B_n$ be left braces. We say that an iterated matched
product of left braces $(\dots(B_1\bowtie B_2)\bowtie\dots)\bowtie
B_n$ is an iterated matched product of left ideals if each $\{
0\}\times\dots\times\{ 0\}\times B_i\times\{ 0\}\times\dots\times\{
0\}$  is a left ideal of it.
\end{definition}

Note that if $B=(\dots(B_1\bowtie B_2)\bowtie\dots)\bowtie B_n$ is
an iterated matched product of left ideals and $\sigma\in\sym_n$, by
Proposition~\ref{inner}, we know that  $B$ is isomorphic to
$(\dots(C_{\sigma(1)}\bowtie C_{\sigma(2)})\bowtie\dots)\bowtie
C_{\sigma(n)}$, where $C_i=\{ 0\}\times\dots\times\{ 0\}\times
B_i\times\{ 0\}\times \dots\times\{ 0\}$  is a left ideal of $B$.
Hence, in an iterated matched product $(\dots(B_1\bowtie
B_2)\bowtie\dots)\bowtie B_n$ of left ideals the order of the
factors is irrelevant and this allows us to write it simply as $B_1\bowtie
B_2\bowtie\dots\bowtie B_n$.

As mentioned before, if $B_1 \bowtie B_2$ is matched product of two left
braces $B_1$ and $B_2$, then it easily is verified that both factors are
left ideals. However, the following example shows that a factor of
an arbitrary iterated matched product is not necessarily a left
ideal of the left brace.  It shows, in particular, that not every
iterated matched product of left braces is a matched product with the defining factors as  left
ideals.

\begin{example}{\rm
Let $A=\mathbb{Z}/(p)$, where $p$ is a prime. Then $A$ is a trivial
brace. Consider the trivial brace $A\times A$ and the maps
$\alpha\colon (A,+)\longrightarrow \Aut (A\times A,+)$, defined by
$\alpha(a)=\alpha_a$ and $\alpha_a(b,c)=(b+ac,c)$ (here $ac$ is the
multiplication in the field $\mathbb{Z}/(p)$). Clearly $\alpha$ is a
homomorphism of groups. The semidirect product  $(A\times A)\rtimes
A$ with respect to $\alpha$ is a left brace with the sum defined
componentwise. This is a particular case of matched product of left
braces. The direct product $A\times A$ is also a particular case of
matched product of left braces. But $\{0\}\times A\times \{0\}$ is
not a left ideal of $(A\times A)\rtimes A$, in fact
$\lambda_{(0,0,1)}(0,1,0)=(\alpha_1(0,1),0)=(1,1,0)$.}
\end{example}

We often will  use the following useful formula, valid in any left brace $B$.
\begin{eqnarray} \label{add_multip}
 b_{1}+\cdots +b_{s}&=&(b_1+\dots +b_{s-1})\lambda_{b_1+\dots +b_{s-1}}^{-1}(b_{s})\nonumber\\
 &=&(b_{1}+\dots+b_{s-2})\lambda_{b_{1}+\dots+b_{s-2}}^{-1}(b_{s-1})\lambda_{b_1+\dots +b_{s-1}}^{-1}(b_{s})
 \nonumber \\
 &\vdots &\nonumber\\
&=&b_{1}\lambda_{b_{1}}^{-1}(b_{2})\lambda_{b_{1}+b_{2}}^{-1}(b_{3})\cdots
\lambda_{b_{1}+\cdots +b_{s-1}}^{-1}(b_{s}),
\end{eqnarray}
 for any $s\geq 1$ and $b_i\in B$, $i=1,\ldots, s$.

\begin{theorem}\label{itlb}
Let $B_1,\dots,B_n$ be left braces with $n\geq 2$. An iterated
matched product $B=(\dots(B_1\bowtie B_2)\bowtie\dots)\bowtie B_n$
of left braces is an iterated matched product of left ideals if and
only if there exist homomorphisms of groups $$\alpha^{(j,i)}\colon
(B_j,\cdot)\longrightarrow \Aut(B_i,+)$$ satisfying the following
conditions:
\begin{itemize}
\item[{\rm (IM1)}]
$\lambda^{(i)}_a\circ\alpha^{(j,i)}_{(\alpha^{(i,j)}_{a})^{-1}(b)}=\alpha^{(j,i)}_b\circ\lambda^{(i)}_{(\alpha^{(j,i)}_b)^{-1}(a)}$
and
\item[{\rm (IM2)}]
$\alpha^{(k,i)}_c\circ\alpha^{(j,i)}_{(\alpha^{(k,j)}_{c})^{-1}(b)}=\alpha^{(j,i)}_b\circ\alpha^{(k,i)}_{(\alpha^{(j,k)}_b)^{-1}(c)}$,
\end{itemize}
 for all $a\in
B_i$, $b\in B_j$, $c\in B_k$, $i,j,k\in \{ 1,2,\dots ,n\}$, $j\neq
i$, $k\neq i$ and $k\neq j$, where $\lambda^{(i)}$ is the lambda map
of the left brace $B_i$ and $\alpha^{(i,j)}(a)=\alpha^{(i,j)}_{(a)}$, and furthermore
$(\dots(B_1\bowtie B_2)\bowtie\dots)\bowtie B_{j+1}$ is the matched product corresponding to the matched pair of left braces
$$((\dots(B_1\bowtie B_2)\bowtie\dots)\bowtie B_j,B_{j+1},\alpha^{(j+1,(1,\dots,j))},\alpha^{((1,\dots,j),j+1)}),$$
where
\begin{eqnarray}\label{cond1}&&\alpha^{((1,\dots,j),j+1)}_{(a_1,\dots,a_{j})}=
\alpha^{(1,j+1)}_{a_{1}}\alpha^{(2,j+1)}_{(\alpha^{(1,2)}_{a_1})^{-1}(a_{2})}\cdots\alpha^{(j,j+1)}_{(\alpha^{((1,\dots,j-1),j)}_{(a_1,\dots,a_{j-1})})^{-1}(a_{j})},\end{eqnarray}
\begin{eqnarray}\label{cond2}&&\alpha^{(j+1,(1,\dots,j))}_{a_{j+1}}(a_1,\dots,a_{j})=(\alpha^{(j+1,1)}_{a_{j+1}}(a_1),\dots,\alpha^{(j+1,j)}_{a_{j+1}}(a_{j})).\end{eqnarray}
\end{theorem}

\begin{proof}
Denote the lambda map of $B_i$ by $\lambda^{(i)}$ and the lambda map
of $(\dots(B_1\bowtie B_2)\bowtie\dots)\bowtie B_{j+1}$ by
$\lambda^{(1,\dots ,j+1)}$, for all $1\leq j<n$. Suppose first that
there exist homomorphisms of groups $\alpha^{(j,i)}\colon
(B_j,\cdot)\longrightarrow \Aut(B_i,+)$ satisfying all the
conditions in the statement. We shall prove that $\{
0\}\times\dots\times\{ 0\}\times B_k\times\{ 0\}\times\dots\times\{
0\}$ is a left ideal of $B$ by induction on $n$. For $n=2$, the
result follows by \cite[Theorem~4.2]{B3}. Suppose that $n>2$ and
that $\{ 0\}\times\dots\times\{ 0\}\times B_i\times\{
0\}\times\dots\times\{ 0\}$  is a left ideal of $((\dots(B_1\bowtie
B_2)\bowtie\dots)\bowtie B_{n-1}$ for all $i=1,\dots,n-1$. Let
$(a_1,\dots ,a_n),(b_1,\dots,b_n)\in B_1\times\dots\times B_n$. We
have that
$$\lambda^{(1,\dots,n)}_{(a_1,\dots,a_n)}(0,\dots,0,b_n)=(0,\dots,0,\alpha^{((1,\dots,n-1),n)}_{(a_1,\dots,a_{n-1})}
\lambda^{(n)}_{(\alpha^{((1,\dots,n-1),n)}_{(a_1,\dots,a_{n-1})})^{-1}(a_n)}(b_n))$$
and for $i<n$
\begin{eqnarray*}
\lefteqn{\lambda^{(1,\dots,n)}_{(a_1,\dots,a_n)}(0,\dots,0,b_i,0,\dots,0)}\\
&&=(\alpha^{(n,(1,\dots,n-1))}_{a_{n}}
\lambda^{(1,\dots ,n-1)}_{(\alpha^{(n,(1,\dots,n-1))}_{a_{n}})^{-1}(a_1,\dots,a_{n-1})}(0,\dots,0,b_i,0,\dots,0),0).\end{eqnarray*}
By induction hypothesis, there exists $c_i\in B_i$, such that
$$\lambda^{(1,\dots ,n-1)}_{(\alpha^{(n,(1,\dots,n-1))}_{a_{n}})^{-1}(a_1,\dots,a_{n-1})}(0,\dots,0,b_i,0,\dots,0)=(0,\dots,0,c_i,0,\dots,0).$$
Hence
\begin{eqnarray*}
\lefteqn{\lambda^{(1,\dots,n)}_{(a_1,\dots,a_n)}(0,\dots,0,b_i,0,\dots,0)}\\
&=&(\alpha^{(n,(1,\dots,n-1))}_{a_{n}}(0,\dots,0,c_i,0,\dots,0),0)\\
&=&((0,\dots,0,\alpha^{(n,i)}_{a_n}(c_i),0,\dots,0),0) \;\; (\mbox{by } (\ref{cond2})).
\end{eqnarray*}
Thus  $\{ 0\}\times\dots\times\{ 0\}\times B_k\times\{
0\}\times\dots\times\{ 0\}$  is a left ideal of $B$ for every
$k=1,\dots,n$. Therefore $B=(\dots(B_1\bowtie
B_2)\bowtie\dots)\bowtie B_n$  is an iterated matched product of
left ideals.

Suppose now that $B=(\dots(B_1\bowtie B_2)\bowtie\dots)\bowtie B_n$ is an iterated matched product of left ideals.
Let $\pi_k$ be the natural projection $\pi_k\colon B_1\times\dots\times B_n\longrightarrow B_k$.
We define
$$\alpha^{(j,i)}_{a_j}(a_i)=\pi_i\lambda^{(1,\dots,n)}_{(0,\dots ,0,a_j,0,\dots,0)}(0,\dots ,0,a_i,0,\dots,0),$$
for all $a_i\in B_i$, $a_j\in B_j$, $i,j\in\{ 1,2,\dots,n\}$ and
$i\neq j$. Note that, since $\{ 0\}\times\dots\times\{ 0\}\times
B_i\times\{ 0\}\times\dots\times\{ 0\}$  is a left ideal of $B$, we
have that
$$\lambda^{(1,\dots,n)}_{(0,\dots ,0,a_j,0,\dots,0)}(0,\dots ,0,a_i,0,\dots,0)=(0,\dots ,0,\alpha^{(j,i)}_{a_j}(a_i),0,\dots,0).$$
Now it is clear that $\alpha^{(j,i)}_{a_j}\in\Aut(B_i,+)$ and
$$(\alpha^{(j,i)}_{a_j})^{-1}(a_i)=\pi_i(\lambda^{(1,\dots,n)}_{(0,\dots ,0,a_j,0,\dots,0)})^{-1}(0,\dots ,0,a_i,0,\dots,0).$$
We shall prove by induction on $n$ that
$$\lambda^{(1,\dots,n)}_{(0,\dots,0,a_i,0,\dots,0)}(0,\dots,0,b_i,0,\dots,0)=(0,\dots,0,\lambda^{(i)}_{a_i}(b_i),0,\dots,0),$$
for all $a_i,b_i\in B_i$ and $i=1,\dots,n$. For $n=2$, this follows easily by the definition of $\lambda^{(1,2)}$. Suppose that $n>2$ and that
$$\lambda^{(1,\dots,n-1)}_{(0,\dots,0,a_i,0,\dots,0)}(0,\dots,0,b_i,0,\dots,0)=(0,\dots,0,\lambda^{(i)}_{a_i}(b_i),0,\dots,0),$$
for all $a_i,b_i\in B_i$ and $i=1,\dots,n-1$. Suppose that $i<n$. In this case
\begin{eqnarray*}\lefteqn{\lambda^{(1,\dots,n)}_{(0,\dots,0,a_i,0,\dots,0)}(0,\dots,0,b_i,0,\dots,0)}\\
&=&(\lambda^{(1,\dots ,n-1)}_{(0,\dots,0,a_i,0,\dots,0)}(0,\dots,0,b_i,0,\dots,0),0)\\
&=&((0,\dots,0,\lambda^{(i)}_{a_i}(b_i),0,\dots,0),0).
\end{eqnarray*}
For $i=n$ we have
$$\lambda^{(1,\dots,n)}_{(0,\dots,0,a_n)}(0,\dots,0,b_n)=(0,\dots ,0,\lambda^{(n)}_{a_n}(b_n)).$$
Note that
\begin{eqnarray*}\lefteqn{(0,\dots,0,a_i,0,\dots,0)(0,\dots,0,b_i,0,\dots,0)}\\
&=&(0,\dots,0,a_i,0,\dots,0)+\lambda^{(1,\dots,n)}_{(0,\dots,0,a_i,0,\dots,0)}(0,\dots,0,b_i,0,\dots,0)\\
&=&(0,\dots,0,a_i+\lambda^{(i)}_{a_i}(b_i),0,\dots,0)\\
&=&(0,\dots,0,a_ib_i,0,\dots,0).\end{eqnarray*}
Hence
$$\lambda^{(1,\dots,n)}_{(0,\dots ,0,a_jb_j,0,\dots,0)}=\lambda^{(1,\dots,n)}_{(0,\dots ,0,a_j,0,\dots,0)}\lambda^{(1,\dots,n)}_{(0,\dots ,0,b_j,0,\dots,0)}.$$
Therefore we get that $\alpha^{(j,i)}_{a_jb_j}=\alpha^{(j,i)}_{a_j}\alpha^{(j,i)}_{b_j}$. Hence the map $\alpha^{(j,i)}\colon (B_j,\cdot)\longrightarrow \Aut(B_i,+)$,
defined by $\alpha^{(j,i)}(a_j)=\alpha^{(j,i)}_{a_j}$, is a homomorphism of groups. Now we shall check condition (IM1).  Note that
\begin{eqnarray*}
\lefteqn{(0,\dots,0,\lambda^{(i)}_{a_i}\alpha^{(j,i)}_{(\alpha^{(i,j)}_{a_i})^{-1}(a_j)}(b_i),0,\dots,0)}\\
&=&\lambda^{(1,\dots,n)}_{(0,\dots,0,a_i,0,\dots,0)}(0,\dots,0,\alpha^{(j,i)}_{(\alpha^{(i,j)}_{a_i})^{-1}(a_j)}(b_i),0,\dots,0)\\
&=&\lambda^{(1,\dots,n)}_{(0,\dots,0,a_i,0,\dots,0)}\lambda^{(1,\dots,n)}_{(0,\dots,0,(\alpha^{(i,j)}_{a_i})^{-1}(a_j),0,\dots,0)}(0,\dots,0,b_i,0,\dots,0)\\
&=&\lambda^{(1,\dots,n)}_{\lambda^{(1,\dots,n)}_{(0,\dots,0,a_i,0,\dots,0)}(0,\dots,0,(\alpha^{(i,j)}_{a_i})^{-1}(a_j),0,\dots,0)}\\
&&\lambda^{(1,\dots,n)}_{(\lambda^{(1,\dots,n)}_{\lambda^{(1,\dots,n)}_{(0,\dots,0,a_i,0,\dots,0)}
(0,\dots,0,(\alpha^{(i,j)}_{a_i})^{-1}(a_j),0,\dots,0)})^{-1}(0,\dots,0,a_i,0,\dots,0)}(0,\dots,0,b_i,0,\dots,0)\\
&=&\lambda^{(1,\dots,n)}_{(0,\dots,0,a_j,0,\dots,0)}\lambda^{(1,\dots,n)}_{(\lambda^{(1,\dots,n)}_{(0,\dots,0,a_j,0,\dots,0)})^{-1}(0,\dots,0,a_i,0,\dots,0)}(0,\dots,0,b_i,0,\dots,0)\\
&=&\lambda^{(1,\dots,n)}_{(0,\dots,0,a_j,0,\dots,0)}\lambda^{(1,\dots,n)}_{(0,\dots,0,(\alpha^{(j,i)}_{a_j})^{-1}(a_i),0,\dots,0)}(0,\dots,0,b_i,0,\dots,0)\\
&=&(0,\dots,0,\alpha^{(j,i)}_{a_j}\lambda^{(i)}_{(\alpha^{(j,i)}_{a_j})^{-1}(a_i)}(b_i),0,\dots,0),
\end{eqnarray*}
where in the third equality \cite[Lemma~2]{CJO2} is used.
Hence $\lambda^{(i)}_{a_i}\alpha^{(j,i)}_{(\alpha^{(i,j)}_{a_i})^{-1}(a_j)}=\alpha^{(j,i)}_{a_j}\lambda^{(i)}_{(\alpha^{(j,i)}_{a_j})^{-1}(a_i)}$ and (IM1) is proved.
Similarly one can check that condition (IM2) is satisfied.
Before proving (\ref{cond1}) we claim that
$$\lambda^{(1,\dots,j+1)}_{(a_1,\dots,a_k,0,\dots,0)}(0,\dots,0,a_{k+1},0,\dots,0)=(0,\dots,0,\alpha^{((1,\dots,k),k+1)}_{(a_1,\dots,a_k)}(a_{k+1}),0,\dots,0),$$
for all $1\leq k\leq j<n$.
We will prove the claim by induction on $j$. For $j=1$, we have
$$\lambda^{(1,2)}_{(a_1,0)}(0,a_2)=(0,\alpha^{(1,2)}_{a_1}(a_2)),$$
by the definition of $\lambda^{(1,2)}$.
Suppose that $j>1$ and that the claim is true for $j-1$. For $k=j$ we have that
$$\lambda^{1,\dots,j+1}_{(a_1,\dots,a_j,0)}(0,\dots,0,a_{j+1})=(0,\dots,0,\alpha^{((1,\dots,j),j+1)}_{(a_1,\dots,a_j)}(a_{j+1})),$$
by the definition of $\lambda^{(1,\dots,j+1)}$. For $k<j$ we have
\begin{eqnarray*}
\lefteqn{
\lambda^{(1,\dots,j+1)}_{(a_1,\dots,a_k,0,\dots,0)}(0,\dots,0,a_{k+1},0,\dots,0)}\\&=&(\lambda^{(1,\dots,j)}_{(a_1,\dots,a_k,0,\dots,0)}(0,\dots,0,a_{k+1},0,\dots,0),0)\\
&=&(0,\dots,0,\alpha^{((1,\dots,k),k+1)}_{(a_1,\dots,a_k)}(a_{k+1}),0,\dots,0),
\end{eqnarray*}
where the first equality is by the definition of $\lambda^{(1,\dots,j+1)}$, and the second is by induction hypothesis. Hence the claim follows.
Now we will prove condition (\ref{cond1}).
We have
\begin{eqnarray*}
\lefteqn{(0,\dots,0,\alpha^{((1,\dots,j),j+1)}_{(a_1,\dots,a_{j})}(a_{j+1}))}\\
&=&\lambda^{(1,\dots,j+1)}_{(a_1,\dots,a_j,0)}(0,\dots,0,a_{j+1})\\
&=&\lambda^{(1,\dots,j+1)}_{(a_1,0,\dots,0)(\lambda^{(1,\dots,j+1)}_{(a_1,0,\dots,0)})^{-1}(0,a_2,0,\dots,0)\cdots
\lambda^{(1,\dots,j+1)}_{(a_1,\dots,a_{j-1},0,0)})^{-1}(0,\dots,0,a_j,0)}(0,\dots,0,a_{j+1})\\
&=&\lambda^{(1,\dots,j+1)}_{(a_1,0,\dots,0)((0,\alpha^{(1,2)}_{a_1})^{-1}(a_2),0,\dots,0)\cdots
(0,\dots,0,\alpha^{((1,\dots,j-1),j)}_{(a_1,\dots,a_{j-1})})^{-1}(a_j),0)}(0,\dots,0,a_{j+1})\\
&=&\lambda^{(1,\dots,j+1)}_{(a_1,0,\dots,0)}\lambda^{(1,\dots,j+1)}_{(0,(\alpha^{(1,2)}_{a_1})^{-1}(a_2),0,\dots,0)}\cdots
\lambda^{(1,\dots,j+1)}_{(0,\dots,0,(\alpha^{((1,\dots,j-1),j)}_{(a_1,\dots,a_{j-1})})^{-1}(a_j),0)}(0,\dots,0,a_{j+1})\\
&=&(0,\dots,0,\alpha^{(1,j+1)}_{a_1}\alpha^{(2,j+1)}_{(\alpha^{(1,2)}_{a_1})^{-1}(a_2)}\cdots
\alpha^{(j,j+1)}_{(\alpha^{((1,\dots,j-1),j)}_{(a_1,\dots,a_{j-1})})^{-1}(a_j)}(a_{j+1})),
\end{eqnarray*}
 where the first equality is by the definition of $\lambda^{(1,\dots,j+1)}$, the second follows from (\ref{add_multip}), the third follows by the claim,
 the fourth is because of the properties of the lambda maps and the last follows because
 $\lambda^{(1,\dots,j+1)}_{(0,\dots,0,a_i,0,\dots,0)}(0,\dots,0,a_k,0,\dots,0)=(0,\dots,0,\alpha^{(i,k)}_{a_i}(a_k),0,\dots,0)$ for all $i\neq k$. Therefore (\ref{cond1}) follows.
 To prove (\ref{cond2}), note that
 \begin{eqnarray*}
 \lefteqn{(\alpha^{(j+1,(1,\dots,j))}_{a_{j+1}}(a_1,\dots,a_j),0)}\\
 &=&\lambda^{(1,\dots,j+1)}_{(0,\dots,0,a_{j+1})}(a_1,\dots,a_j,0)\\
 &=&\sum_{i=1}^j \lambda^{(1,\dots,j+1)}_{(0,\dots,0,a_{j+1})}(0,\dots,0,a_i,0,\dots,0)\\
 &=&\sum_{i=1}^j (0,\dots,0,\alpha^{(j+1,i)}_{a_{j+1}}(a_i),0,\dots,0)\\
 &=&(\alpha^{(j+1,1)}_{a_{j+1}}(a_1),\alpha^{(j+1,2)}_{a_{j+1}}(a_2),\dots,\alpha^{(j+1,j)}_{a_{j+1}}(a_j),0).
 \end{eqnarray*}
 Hence (\ref{cond2}) follows. This finishes the proof of the theorem.
\end{proof}

Using the notation of Theorem~\ref{itlb}, for an iterated matched product of left ideals $B=(\dots(B_1\bowtie B_2)\bowtie\dots)\bowtie B_n$, it can be checked that
the $i$-th component of $\lambda^{(1,\dots,n)}_{(a_1,\dots,a_n)}(b_1,\dots,b_n)$ is of the form
\begin{eqnarray}\label{form_lambda}
\alpha^{(1,i)}_{a_1}\alpha^{(2,i)}_{(\alpha^{(1,i)}_{a_1})^{-1}(a_2)}\cdots \alpha^{(i-1,i)}_{(\alpha^{((1,\dots,i-2),i-1)}_{(a_1,\dots,a_{i-2})})^{-1}(a_{i-1})}
\lambda^{(i)}_{(\alpha^{((1,\dots,i-1),i)}_{(a_1,\dots,a_{i-1})})^{-1}(a_{i})}\nonumber\\
\quad\cdot\alpha^{(i+1,i)}_{(\alpha^{((1,\dots,i),i+1)}_{(a_1,\dots,a_{i})})^{-1}(a_{i+1})}\cdots \alpha^{(n,i)}_{(\alpha^{((1,\dots,n-1),n)}_{(a_1,\dots,a_{n-1})})^{-1}(a_{n})}(b_i).
\end{eqnarray}
Note that one can interpret $B_1,\dots, B_n$ as left ideals of $B$ such that the additive group of $B$ is the direct sum of the additive groups of the left ideals $B_i$.
Then $(a_1,\cdots ,a_n)$ corresponds to $a_1+\dots +a_n$, the maps $\alpha$
correspond to some restrictions of the lambda map of $B$ and formula (\ref{form_lambda}) follows from (\ref{add_multip}).

In the remainder of this section we focus on simplicity of left braces that are iterated matched products of left ideals.

 We will use the following easy but
useful result.

\begin{lemma}\label{ideal}  If $I$ is an ideal of a left brace $B$, then
$(\lambda_b-\id)(a)\in I$, for all $a\in B$ and $b\in I$.
\end{lemma}

\begin{proof}
Let $a\in B$ and $b\in I$. Then $ (\lambda_b-\id)(a)=ba-b-a
=\lambda_a(a^{-1}ba)-b\in I$, so the assertion follows.
\end{proof}
Let $B$ be an iterated matched product of its left ideals
$B_{1},\ldots, B_{s}$ of relatively prime orders. Consider the
oriented graph $\Gamma (B)=(V,E)$, defined as follows.
The set of vertices
$V=\{1,\ldots, s\}$   and $(i,j)\in E$ is an
edge if the corresponding map $\alpha^{(i,j)}:B_{i}\longrightarrow
\Aut (B_{j},+)$ is nontrivial. We call $\Gamma (B)$ the graph of
(nontrivial) actions of $B$.

\begin{theorem} \label{graph_simple}
With the above notation and assumptions, assume that every $B_{i}$ is a simple left
brace. Then the left brace $B$ is simple if and only if
$\Gamma=\Gamma(B)$ contains a full (oriented) cycle, i.e. a cycle that contains all vertices.
\end{theorem}
\begin{proof}
Suppose $\Gamma $ contains a full cycle. Let $I$ be a nonzero ideal
of $B$. Choose a nonzero element $b_1+\cdots +b_s\in I$, with
$b_{i}\in B_{i}$. Since $B$ is a direct sum of its additive
subgroups $B_{i}$ of relatively prime orders, it is easy to see that
$0\neq b_i\in I\cap B_{i}$ for some $i$. Let $j\in V$ be such that
$(i,j)\in E$. Hence there exists an element $a_{j}\in B_{j}$ such
that $\alpha^{(i,j)}_{b}(a_{j})\neq a_{j}$. Recall that
$\alpha^{(i,j)}_{a_{i}}(a_j)= a_{i}a_{j}-a_{i}=
\lambda_{a_{i}}(a_{j})$. So, by Lemma~\ref{ideal},  $0\neq
\alpha^{(i,j)}_{b}(a_{j})-a_{j}\in I\cap B_{j}$. It follows that
$B_{j}\subseteq I$ because $B_{j}$ is a simple left brace. Since
$\Gamma $ contains a full cycle, this easily implies that
$B_{k}\subseteq I$ for every $k$, and consequently $I=B$.

Conversely, assume that $\Gamma $ contains no full cycle. Then there
exists $i$ such that  $W=\{ k\in V: \mbox{ there exists a path in }
\Gamma \mbox{ from } i \mbox{ to } k\} \neq V$. Let $I=\sum_{j\in W}
B_{j}$. Clearly, $I$ is a left ideal of $B$. Hence $I$ is
$\lambda$-invariant. We will check that $I$ is an ideal of $B$. Then
the result follows.  Let $b\in B$. Write $b=b_1+\cdots +b_s$, with
$b_i\in B_i$. From (\ref{add_multip}) we know that $b=c_{1}\cdots
c_{s}$ for some $c_{i}\in B_{i}$. Let $a\in B_{j}$ for $j\in W$, so
that $a\in I$. Since $\lambda_{c} (c^{-1}ac) = \lambda_{a}(c) - c
+a$, we get that $c^{-1}ac = \lambda_{c}^{-1}(\lambda_{a}(c)-c+a)$.
If $c\in B_{k}$ for some $k\notin W$ then $\lambda_{a}(c)=c$, so
that $c^{-1}ac =\lambda_{c}^{-1}(a)\in I$. On the other hand, if
$c\in B_{k}$ for some $k\in W$, then $c,a\in I$ and thus
also $c^{-1}ac\in I$.

We know that $a=\sum_{j\in W}a_{j}$, with $a_{j}\in B_{j}, j\in W$.
Therefore, again by (\ref{add_multip}), we also have
$a=a_{j_{1}}\cdots a_{j_{k}}$, where $j_{1},\ldots, j_{k}\in W$. Now
$c^{-1}ac= c^{-1}a_{j_{1}}c c^{-1}a_{j_{2}}c \cdots
c^{-1}a_{j_{k}}c\in I$ for every $c\in B_{i}$ and any $i$. Since
$b=c_{1}\cdots c_{s}$, we get that $b^{-1}ab\in I$. Hence, $I$ is an
ideal of $B$ and the result follows.
\end{proof}

Notice that the proof of the necessity in the above theorem does not
require the hypothesis that $B_{i}$ are simple left braces. Hence, the existence of a full oriented cycle in $\Gamma (B)$ is a necessary condition for $B$ to be simple.

We shall see in Example~\ref{ex} that the following result provides
an effective way for constructing matched products of left braces.

\begin{proposition}\label{morphism}
Let  $B_{1},\ldots, B_{s}$ be left braces. Assume that
$\alpha^{(i,j)}: (B_{i}, \cdot)\longrightarrow \Aut (B_{j},+,\cdot)$
are group homomorphisms, for all $i,j\in \{1,\ldots ,s\}$, $i\neq
j$. Assume also that
\begin{itemize}
\item[(i)]
$\alpha^{(i,j)}_{\alpha^{(k,i)}_{a_{k}}(a_{i})}= \alpha^{(i,j)}_{a_{i}}$,
\item[(ii)] $\alpha^{(j,i)}_{a_{j}} \alpha^{(k,i)}_{a_{k}} = \alpha^{(k,i)}_{a_{k}} \alpha^{(j,i)}_{a_{j}}$,
\end{itemize}
 for all $i,j,k\in \{1,\ldots, s\}$ with $i\neq j$ and $i\neq k$, where $a_{m}\in B_{m}$ for every $m$.
 Then the maps $\alpha^{(i,j)}$ satisfy conditions (IM1) and (IM2) and defining $\alpha^{(j+1, (1,\ldots, j))}$
 and $\alpha^{((1,\ldots, j),j+1)}$ as in (\ref{cond1}) and (\ref{cond2}), we get an iterated matched product of left ideals $B_{1}\bowtie \dots \bowtie B_{s} $.
 \end{proposition}
\begin{proof}
It is enough to verify
conditions (IM1) and (IM2) stated in Theorem~\ref{itlb}.
\begin{eqnarray*}
\lambda^{(i)}_{a_i}\alpha^{(j,i)}_{(\alpha^{(i,j)}_{a_i})^{-1}(a_j)}(b_i)&=&\lambda^{(i)}_{a_i}\alpha^{(j,i)}_{a_j}(b_i)\\
&=&a_i\alpha^{(j,i)}_{a_j}(b_i)-a_i\\
&=&\alpha^{(j,i)}_{a_j}((\alpha^{(j,i)}_{a_j})^{-1}(a_i)b_i-(\alpha^{(j,i)}_{a_j})^{-1}(a_i))\\
&=&\alpha^{(j,i)}_{a_j}\circ\lambda^{(i)}_{(\alpha^{(j,i)}_{a_j})^{-1}(a_i)}(b_i),
\end{eqnarray*}
for $a_i,b_i\in B_i$ and $a_j\in B_j$, where $\lambda^{(i)}$ is the lambda map of $B_i$.
Thus (IM1) follows.
Now we verify condition (IM2).
\begin{eqnarray*}
\lefteqn{
\alpha^{(k,i)}_{a_k}\circ\alpha^{(j,i)}_{(\alpha^{(k,j)}_{a_k})^{-1}(a_j)}}\\
&=&\alpha^{(k,i)}_{a_k}\circ\alpha^{(j,i)}_{a_j}
\;=\;\alpha^{(j,i)}_{a_j}\circ\alpha^{(k,i)}_{a_k}
\;=\;\alpha^{(j,i)}_{a_j}\circ\alpha^{(k,i)}_{(\alpha^{(j,k)}_{a_j})^{-1}(a_k)},
\end{eqnarray*}
for $a_i\in B_i$, $a_j\in B_j$ and $a_k\in B_k$.
Thus the result follows.\end{proof}

\section{Constructions of simple braces}\label{simple}

In this section, we first present  a family of left braces with
trivial socle, that generalizes the family presented by
Heged\H{u}s \cite{H} and Catino and Rizzo in \cite{CR}.
Then we use it to construct a broad family of simple left braces.

Let $p$ be a prime number, and let $r,n$ be positive integers.
Assume $Q$ is a quadratic form  over $(\mathbb{Z}/(p^r))^n$
(considered as a free module over the ring $\mathbb{Z}/(p^r)$) and
suppose $f$ is an element in the orthogonal group of $Q$ (that is,
an element $f\in \aut((\mathbb{Z}/(p^r))^n)$  such that
$Q(f(v))=Q(v)$ for any $v\in (\mathbb{Z}/(p^r))^n$).  Assume that
$f$ has order $p^{r'}$ for some $0\leq r'\leq r$. Consider the
additive abelian group $A=(\mathbb{Z}/(p^r))^{n+1}$. The elements of
$A$ will be written in the form $(\vec{x},\mu)$, with $\vec{x}\in
(\mathbb{Z}/(p^r))^n$ and $\mu\in \mathbb{Z}/(p^r)$. Consider the
maps $\lambda_{(\vec{x},\mu)}\colon A\longrightarrow A$ defined by
\begin{eqnarray}\label{lamb}
\lambda_{(\vec{x},\mu)}(\vec{y},\mu'):=(f^{q(\vec{x},\mu)}(\vec{y}),
\mu'+b(\vec{x},f^{q(\vec{x},\mu)}(\vec{y}))),
\end{eqnarray}
for $(\vec{x},\mu),(\vec{y},\mu')\in A$, where
$q(\vec{x},\mu):=\mu-Q(\vec{x})$, and $b$ is the bilinear form
$b(\vec{x},\vec{y}):=Q(\vec{x}+\vec{y})-Q(\vec{x})-Q(\vec{y})$
associated to $Q$. Note that $\lambda_{(\vec{x},\mu)}$ is
well-defined since $q$ takes values in $\mathbb{Z}/(p^r)$ and $f$ is
of order $p^{r'}$, for some $0\leq r'\leq r$.

Recall that $Q$ is non-degenerate if and only if the matrix  of $b$
in the standard basis of $(\Z/(p^r))^{n}$ is invertible.

\begin{theorem}\label{hegedus}
The abelian group $A$ has a structure of a left brace with lambda map
defined in (\ref{lamb}) and with multiplication given by $a\cdot b = a+\lambda_{a}(b)$.
Moreover, if $Q$ is non-degenerate, then the socle
of this  left brace is
$$\soc (A)=\{ (\vec{0},\mu)\mid \mu\in p^{r'}\mathbb{Z}/(p^r)\}.$$
In particular, if $r'=r$, then the socle of this left brace is zero.
\end{theorem}
\begin{proof}
Since $f$ is bijective, it is clear that  $\lambda_{(\vec{x},\mu)}$
is bijective. By the definition of $\lambda_{(\vec{x},\mu)}$, it
also is clear that it is an automorphism of the abelian group $A$.

To prove the first part of the result, by \cite[Lemma~2.6]{B3}, it
is enough to check that
$\lambda_{(\vec{x},\mu)}\lambda_{(\vec{y},\mu')}=\lambda_{(\vec{x},\mu)+\lambda_{(\vec{x},\mu)}(\vec{y},\mu')}$.
On one side,
\begin{align}
\lefteqn{\lambda_{(\vec{x},\mu)}\lambda_{(\vec{y},\mu')}(\vec{z},\eta)} \nonumber\\
&=\lambda_{(\vec{x},\mu)}(f^{q(\vec{y},\mu')}(\vec{z}),
\eta+b(\vec{y},f^{q(\vec{y},\mu')}(\vec{z}))) \nonumber\\
&=(f^{q(\vec{x},\mu)+q(\vec{y},\mu')}(\vec{z}),
\eta+b(\vec{y},f^{q(\vec{y},\mu')}(\vec{z}))+b(\vec{x},f^{q(\vec{x},\mu)+q(\vec{y},\mu')}(\vec{z}))). \label{newequation}
\end{align}
Note that
\begin{align*}
\lefteqn{q((\vec{x},\mu)+\lambda_{(\vec{x},\mu)}(\vec{y},\mu'))}\\
&=q(\vec{x}+f^{q(\vec{x},\mu)}(\vec{y}), \mu+\mu'+b(\vec{x},f^{q(\vec{x},\mu)}(\vec{y})))\\
&=\mu+\mu'+b(\vec{x},f^{q(\vec{x},\mu)}(\vec{y}))-Q(\vec{x}+f^{q(\vec{x},\mu)}(\vec{y}))\\
&=\mu+\mu'+b(\vec{x},f^{q(\vec{x},\mu)}(\vec{y}))-Q(\vec{x})-Q(f^{q(\vec{x},\mu)}(\vec{y}))
 -b(\vec{x},f^{q(\vec{x},\mu)}(\vec{y}))\\
&=\mu+\mu'-Q(\vec{x})-Q(\vec{y})\\
&=q(\vec{x},\mu)+q(\vec{y},\mu').
\end{align*}
Hence
\begin{eqnarray}\label{q}
&&q((\vec{x},\mu)+\lambda_{(\vec{x},\mu)}(\vec{y},\mu'))=q(\vec{x},\mu)+q(\vec{y},\mu').
\end{eqnarray}

On the other side we have
\begin{align*}
\lefteqn{\lambda_{(\vec{x},\mu)+\lambda_{(\vec{x},\mu)}(\vec{y},\mu')}(\vec{z},\eta)}\\
&=(f^{q((\vec{x},\mu)+\lambda_{(\vec{x},\mu)}(\vec{y},\mu'))}(\vec{z}),\eta+b(\vec{x}+f^{q(\vec{x},\mu)}(\vec{y}),f^{q((\vec{x},\mu)+\lambda_{(\vec{x},\mu)}(\vec{y},\mu))}(\vec{z})))\\
&=(f^{q(\vec{x},\mu)+q(\vec{y},\mu')}(\vec{z}),\eta+b(\vec{x}+f^{q(\vec{x},\mu)}(\vec{y}),f^{q(\vec{x},\mu)+q(\vec{y},\mu')}(\vec{z})))\quad
(\mbox{by
(\ref{q})})\\
&=(f^{q(\vec{x},\mu)+q(\vec{y},\mu')}(\vec{z}),\eta+b(f^{q(\vec{x},\mu)}(\vec{y}),f^{q(\vec{x},\mu)+q(\vec{y},\mu')}(\vec{z}))+b(\vec{x},f^{q(\vec{x},\mu)+q(\vec{y},\mu')}(\vec{z})))\\
&=(f^{q(\vec{x},\mu)+q(\vec{y},\mu')}(\vec{z}),\eta+b(\vec{y},f^{q(\vec{y},\mu')}(\vec{z}))+b(\vec{x},f^{q(\vec{x},\mu)+q(\vec{y},\mu')}(\vec{z}))).
\end{align*}
Therefore,  by (\ref{newequation}),
$\lambda_{(\vec{x},\mu)}\lambda_{(\vec{y},\mu')}=\lambda_{(\vec{x},\mu)+\lambda_{(\vec{x},\mu)}(\vec{y},\mu')}$,
as desired.

To prove the second part of the statement, assume that $Q$ is non-degenerate. Let
$(\vec{x},\mu)$ be an element of the socle of the left brace $A$.
Then $f^{q(\vec{x},\mu)}=\id$ and
$b(\vec{y},f^{q(\vec{x},\mu)}(\vec{x}))=0$ for all $\vec{y}$.
Since $Q$ is non-degenerate, $\vec{x}=0$. On the other hand, since
$f^{q(\vec{x},\mu)}=\id$  and $f$ has order $p^{r'}$, we have
$\mu=\mu-Q(\vec{x})=q(\vec{x},\mu)\in p^{r'}\mathbb{Z}/(p^r)$.
 Therefore, the
result follows.
\end{proof}

\noindent {\bf Notation.}  The left brace described in
Theorem~\ref{hegedus} is denoted by $H(p^r,n,Q,f)$.

\vspace{10pt}
Let $R$ be a ring. For any matrix $A$ over $R$, we denote by $A^t$
the transpose of $A$. Sometimes we identify $R^n$ with the row
matrices of length $n$ over $R$. So, for $x\in R^n$, $x^t$ is the
column transpose of the row $x$.

Now we shall construct iterated matched products of left braces of
the form $H(p^r,n,Q,f)$ and, as a consequence, we will give some new constructions of
finite simple left braces. To do so we will make use  of the
existence of elements $C$ of order $p^r$ in $\GL_n(\Z/({q}^{s}))$
for two different primes $p$ and $q$. Note that the natural image of
$C$ in $\GL_n(\Z/({q}))$ also has order $p^r$ and therefore $p^{r}$
has to divide $(q^{n}-1) \cdots (q^{n}-q^{n-1})$. In particular,
$p\mid q^{t}-1$ for some $1\leq t \leq n$. In light of the necessary
condition for the existence of finite simple left braces mentioned
in the introduction this is a natural assumption which will be
implicitly showing up throughout the paper.

We will fix some notation. Let $s$ be an integer greater than $1$ and
let $p_1,p_2,\ldots ,p_s$ be different prime numbers.  Assume that
$p_1,p_2,\dots ,p_{s-1}$ are odd. For $1\leq i\leq s$, assume that
finite left braces $H_i=H(p_i^{r_i},n_i,Q_i,f_i)$ are constructed as
in Theorem~\ref{hegedus},  with additive groups $(\mathbb
Z/(p_i^{r_i}))^{n_i+1}$ ($r_i$ and $n_i$ are positive integers) and
with the corresponding lambda map defined by
\begin{equation} \label{lambda}
\lambda^{(i)}_{(\vec{x}_i,\mu_i)}(\vec{y}_i,\mu'_i)= (f_{i}^{q_{i}(\vec{x}_i,\mu_i)}(\vec{y}_i), \mu'_i+b_{i}(\vec{x}_i,f_{i}^{q_i(\vec{x}_i,\mu_i)}(\vec{y}_i))),
\end{equation}
where
\begin{itemize}
\item[-] $Q_i$ is a non-degenerate quadratic form over $(\mathbb
Z/(p_i^{r_i}))^{n_i}$,
\item[-]
$f_i$ is an element of  order $p_i^{r'_i}$ in the orthogonal
group determined by $Q_i$,  for some $0\leq r'_i\leq r_i$,
\item[-]
$q_{i}(\vec{x}_i,\mu_i)= \mu_i - Q_{i}(\vec{x}_i)$ (with $\mu_i\in
\mathbb{Z}/(p_{i}^{r_i})$),
\item[-]
$b_i(\vec{x}_i,\vec{y}_i) = Q_i(\vec
{x}_i+\vec{y}_i)-Q_i(\vec{x}_i)-Q_i(\vec{y}_i)$.
\end{itemize}
For $1\leq i< s$, suppose   $c_{i}$ is an element of order
$p_{i+1}^{r_{i+1}}$  in the orthogonal group determined by
$Q_i$, $c_s$ is an element of order $p_{1}^{r_{1}}$ of
$\aut((\mathbb{Z}/(p_s^{r_s}))^{n_s})$ and $v_{s}\in (\mathbb
Z/(p_{s}^{r_{s}}))^{n_{s}}$, such that
\begin{eqnarray}\label{Qj+1}
&&Q_{s}(c_{s}(\vec{x}))=Q_{s}(\vec{x})+v_{s}\vec{x}^t,\end{eqnarray}
and $$f_ic_{i}=c_{i}f_i,$$
 for  $1\leq i\leq s$. For $1\leq i, j\leq
s$, define the maps
$$\alpha^{(j,i)}: (H_{j},\cdot) \longrightarrow \Aut (H_i,+):(\vec{x}_{j},\mu_{j})\mapsto
\alpha^{(j,i)}_{(\vec{x}_{j},\mu_{j})},$$ with
\begin{eqnarray*}&&\alpha^{(k+1,k)}_{(\vec{x}_{k+1},\mu_{k+1})}(\vec{x}_k,\mu_k)=(c_{k}^{q_{k+1}(\vec{x}_{k+1},\mu_{k+1})}(\vec{x}_k),
\mu_k), \mbox{ for }1\leq k<s,\\
&&\alpha^{(1,s)}_{(\vec{x}_1,\mu_1)}(\vec{x}_{s},\mu_{s})
=(c_{s}^{q_{1}(\vec{x}_1,\mu_1)}(\vec{x}_{s}), \mu_{s}+
v_{s}((\id+c_{s}+\dots
+c_{s}^{q_{1}(\vec{x}_1,\mu_1)-1})(\vec{x}_{s}))^t),
\end{eqnarray*}
and $\alpha^{(j,i)}_{(\vec{x}_{j},\mu_{j})}=\id_{H_i}$ otherwise.
It
is easy to check that
$$b_{s}(c_{s}(\vec{x}_{s}),c_{s}(\vec{y}_{s}))=b_{s}(\vec{x}_{s},\vec{y}_{s}).$$
Thus, if  $p_{s}\neq 2$, then $v_{s}=0$.  Note that
$\alpha^{(k+1,k)}_{(\vec{x}_{k+1},\mu_{k+1})}$ is well-defined since
$q_{k+1}$ takes values in $\mathbb{Z}/(p_{k+1}^{r_{k+1}})$, and the
$c_{k}$ are of order $p_{k+1}^{r_{k+1}}$. Similarly
$\alpha^{(1,s)}_{(\vec{x}_1,\mu_1)}$ is well-defined since $q_{1}$
takes values in $\mathbb{Z}/(p_{1}^{r_1})$, and $c_{s}$ is of order
$p_1^{r_1}$. Note also that
$$q_i\colon
(H_i,\cdot)\longrightarrow (\mathbb{Z}/(p_{i}^{r_i}),+)$$ is a
homomorphism of groups,  because
$$(\vec{z},\nu)\cdot(\vec{t},\nu')=(\vec{z},\nu)+\lambda^{(i)}_{(\vec{z},\nu)}(\vec{t},\nu')$$
and by (\ref{q}),
$$q_i((\vec{z},\nu)+\lambda^{(i)}_{(\vec{z},\nu)}(\vec{t},\nu'))=q_i(\vec{z},\nu)+q_i(\vec{t},\nu').$$
Therefore, each $\alpha^{(j,i)}$ is a group homomorphism.

\begin{lemma}\label{q2primes}
With the above notation, for $1\leq i, j\leq s$ and $(\vec{x}_i,\mu_i)\in H_i$,
$$q_i(\alpha^{(j,i)}_{(\vec{x}_{j},\mu_{j})}(\vec{x}_i,\mu_i))=q_i(\vec{x}_i,\mu_i).$$
\end{lemma}

\begin{proof}
Let $(\vec{x}_k,\mu_k)\in H_k$, for $k=1,\dots ,s$. For $i=1,\dots
,s-1$, we have
\begin{eqnarray*}q_i(\alpha^{(i+1,i)}_{(\vec{x}_{i+1},\mu_{i+1})}(\vec{x}_i,\mu_i))&=&q_i(c_{i}^{q_{i+1}(\vec{x}_{i+1},\mu_{i+1})}(\vec{x}_i),\mu_i)\\
&=&\mu_i-Q_i(c_{i}^{q_{i+1}(\vec{x}_{i+1},\mu_{i+1})}(\vec{x}_i))\\
&=&\mu_i-Q_i(\vec{x}_i)\\
&=&q_i(\vec{x}_i,\mu_i).
\end{eqnarray*}
On the other hand
\begin{eqnarray*}
\lefteqn{q_{s}(\alpha^{(1,s)}_{(\vec{x}_1,\mu_1)}(\vec{x}_{s},\mu_{s}))}\\
&=&q_{s}(c_{s}^{q_{1}(\vec{x}_{1},\mu_{1})}(\vec{x}_{s}),\mu_{s}+v_{s}((\id+c_{s}+\dots
+c_{s}^{q_{1}(\vec{x}_{1},\mu_{1})-1})(\vec{x}_{s}))^t)\\
&=&\mu_{s}+v_{s}((\id+c_{s}+\dots
+c_{s}^{q_{1}(\vec{x}_{1},\mu_{1})-1})(\vec{x}_{s}))^t -Q_{s}(c_{s}^{q_{1}(\vec{x}_{1},\mu_{1})}(\vec{x}_{s}))\\
&=&\mu_{s}-Q_{s}(\vec{x}_{s})\qquad(\mbox{by (\ref{Qj+1})})\\
&=&q_{s}(\vec{x}_{s},\mu_{s}).
\end{eqnarray*}
Therefore, the result follows.
\end{proof}

It follows from the definitions that the map
$\alpha^{(j,i)}_{(\vec{x}_{j},\mu_{j})}$ does not depend directly on
the element $(\vec{x}_{j},\mu_{j})$ but just on the value $q_j
(\vec{x}_{j},\mu_{j})$. Therefore Lemma~\ref{q2primes} leads to the
following consequence.

\begin{lemma}\label{qsprimes}
With the above notation, we have
$$\alpha^{(j,i)}_{\alpha^{(k,j)}_{(\vec{x}_{k},\mu_{k})}(\vec{x}_{j},\mu_{j})}=\alpha^{(j,i)}_{(\vec{x}_{j},\mu_{j})}.$$
\end{lemma}

\begin{lemma}\label{auto}
With the above notation, we have that
$\alpha^{(j,i)}_{(\vec{x}_{j},\mu_{j})}\in \aut(H_i,+,\cdot)$.
\end{lemma}
\begin{proof}
Since $\alpha^{(j,i)}_{(\vec{x}_{j},\mu_{j})}\in \aut (H_i,+)$, to
prove the result it is enough  to show that
\begin{eqnarray}\label{alpha}
&&\alpha^{(j,i)}_{(\vec{x}_{j},\mu_{j})}\lambda^{(i)}_{(\vec{x}_i,\mu_i)}(\vec{y}_i,\mu'_i)
=\lambda^{(i)}_{\alpha^{(j,i)}_{(\vec{x}_{j},\mu_{j})}(\vec{x}_i,\mu_i)}\alpha^{(j,i)}_{(\vec{x}_{j},\mu_{j})}(\vec{y}_i,\mu'_i).\end{eqnarray}
For $1\leq i<s$ we have that
\begin{eqnarray*}\lefteqn{
\alpha^{(i+1,i)}_{(\vec{x}_{i+1},\mu_{i+1})}\lambda^{(i)}_{(\vec{x}_i,\mu_i)}(\vec{y}_i,\mu'_i)}\\
&=&
\alpha^{(i+1,i)}_{(\vec{x}_{i+1},\mu_{i+1})}(f_i^{q_i(\vec{x}_i,\mu_i)}(\vec{y}_i),\mu'_i+b_i(\vec{x}_i,f_i^{q_i(\vec{x}_i,\mu_i)}(\vec{y}_i)))\\
&=&
(c_i^{q_{i+1}(\vec{x}_{i+1},\mu_{i+1})}f_i^{q_i(\vec{x}_i,\mu_i)}(\vec{y}_i),\mu'_i+b_i(\vec{x}_i,f_i^{q_i(\vec{x}_i,\mu_i)}(\vec{y}_i))),
\end{eqnarray*}
and
\begin{eqnarray*}\lefteqn{
\lambda^{(i)}_{\alpha^{(i+1,i)}_{(\vec{x}_{i+1},\mu_{i+1})}(\vec{x}_i,\mu_i)}\alpha^{(i+1,i)}_{(\vec{x}_{i+1},\mu_{i+1})}(\vec{y}_i,\mu'_i)}\\
&=&
\lambda^{(i)}_{\alpha^{(i+1,i)}_{(\vec{x}_{i+1},\mu_{i+1})}(\vec{x}_i,\mu_i)}(c_i^{q_{i+1}(\vec{x}_{i+1},\mu_{i+1})}(\vec{y}_i),\mu'_i)\\
&=&
(f_i^{q_{i}(\alpha^{(i+1,i)}_{(\vec{x}_{i+1},\mu_{i+1})}(\vec{x}_i,\mu_i))}c_i^{q_{i+1}(\vec{x}_{i+1},\mu_{i+1})}(\vec{y}_i),\mu'_i\\
&&\qquad +b_i(c_i^{q_{i+1}(\vec{x}_{i+1},\mu_{i+1})}(\vec{x}_{i}),f_i^{q_{i}(\alpha^{(i+1,i)}_{(\vec{x}_{i+1},\mu_{i+1})}(\vec{x}_i,\mu_i))}c_i^{q_{i+1}(\vec{x}_{i+1},\mu_{i+1})}(\vec{y}_i)))\\
&=&
(c_i^{q_{i+1}(\vec{x}_{i+1},\mu_{i+1})}f_i^{q_{i}(\vec{x}_i,\mu_i)}(\vec{y}_i),\mu'_i\\
&&\qquad +b_i(\vec{x}_{i},f_i^{q_{i}(\vec{x}_i,\mu_i)}(\vec{y}_i)))
\quad (\mbox{by Lemma~\ref{q2primes} and because }f_ic_i=c_if_i).
\end{eqnarray*}
Hence
\begin{eqnarray*}
&&\alpha^{(i+1,i)}_{(\vec{x}_{i+1},\mu_{i+1})}\lambda^{(i)}_{(\vec{x}_i,\mu_i)}(\vec{y}_i,\mu'_i)
=\lambda^{(i)}_{\alpha^{(i+1,i)}_{(\vec{x}_{i+1},\mu_{i+1})}(\vec{x}_i,\mu_i)}\alpha^{(i+1,i)}_{(\vec{x}_{i+1},\mu_{i+1})}(\vec{y}_i,\mu'_i).\end{eqnarray*}
Because $f_s c_s =c_s f_s$ and since $f_s$ is orthogonal with
respect to $Q_s$, using (\ref{Qj+1}), one easily verifies that
$v_s f_s (\vec{y}^{t}) =v_s \vec{y}^{t}$. Hence, we also have that
\begin{eqnarray*}\lefteqn{
\alpha^{(1,s)}_{(\vec{x}_{1},\mu_{1})}\lambda^{(s)}_{(\vec{x}_s,\mu_s)}(\vec{y}_s,\mu'_s)}\\
&=&
\alpha^{(1,s)}_{(\vec{x}_{1},\mu_{1})}(f_s^{q_s(\vec{x}_s,\mu_s)}(\vec{y}_s),\mu'_s+b_s(\vec{x}_s,f_s^{q_s(\vec{x}_s,\mu_s)}(\vec{y}_s)))\\
&=&
(c_s^{q_{1}(\vec{x}_{1},\mu_{1})}f_s^{q_s(\vec{x}_s,\mu_s)}(\vec{y}_s),\mu'_s+b_s(\vec{x}_s,f_s^{q_s(\vec{x}_s,\mu_s)}(\vec{y}_s))\\
&&\qquad +v_s((\id+c_s+\dots
+c_s^{q_{1}(\vec{x}_{1},\mu_{1})-1})(f_s^{q_s(\vec{x}_s,\mu_s)}(\vec{y}_s)))^t)\\
&=&
(c_s^{q_{1}(\vec{x}_{1},\mu_{1})}f_s^{q_s(\vec{x}_s,\mu_s)}(\vec{y}_s),\mu'_s+b_s(\vec{x}_s,f_s^{q_s(\vec{x}_s,\mu_s)}(\vec{y}_s))\\
&&\qquad +v_s((\id+c_s+\dots
+c_s^{q_{1}(\vec{x}_{1},\mu_{1})-1})(\vec{y}_s))^t)\quad(\mbox{by (\ref{Qj+1})})\\
\end{eqnarray*}
and
\begin{eqnarray*}\lefteqn{
\lambda^{(s)}_{\alpha^{(1,s)}_{(\vec{x}_{1},\mu_{1})}(\vec{x}_s,\mu_s)}\alpha^{(1,s)}_{(\vec{x}_{1},\mu_{1})}(\vec{y}_s,\mu'_s)}\\
&=&
\lambda^{(s)}_{\alpha^{(1,s)}_{(\vec{x}_{1},\mu_{1})}(\vec{x}_s,\mu_s)}(c_s^{q_{1}(\vec{x}_{1},\mu_{1})}(\vec{y}_s),\mu'_s\\
&&\qquad +v_s((\id+c_s+\dots
+c_s^{q_{1}(\vec{x}_{1},\mu_{1})-1})(\vec{y}_s))^t)\\
&=&
(f_s^{q_{s}(\alpha^{(1,s)}_{(\vec{x}_{1},\mu_{1})}(\vec{x}_s,\mu_s))}c_s^{q_{1}(\vec{x}_{1},\mu_{1})}(\vec{y}_s),\mu'_s\\
&&\qquad +v_s((\id+c_s+\dots
+c_s^{q_{1}(\vec{x}_{1},\mu_{1})-1})(\vec{y}_s))^t)\\
&&\qquad +b_s(c_s^{q_{1}(\vec{x}_{1},\mu_{1})}(\vec{x}_{s}),f_s^{q_{s}(\alpha^{(1,s)}_{(\vec{x}_{1},\mu_{1})}(\vec{x}_s,\mu_s))}c_s^{q_{1}(\vec{x}_{1},\mu_{1})}(\vec{y}_s)))\\
&=&
(c_s^{q_{1}(\vec{x}_{1},\mu_{1})}f_s^{q_{s}(\vec{x}_s,\mu_s)}(\vec{y}_s),\mu'_s\\
&&\qquad +v_s((\id+c_s+\dots
+c_s^{q_{1}(\vec{x}_{1},\mu_{1})-1})(\vec{y}_s))^t)\\
&&\qquad +b_s(\vec{x}_{s},f_s^{q_{s}(\vec{x}_s,\mu_s)}(\vec{y}_s)))
\quad (\mbox{by Lemma~\ref{q2primes} and because }f_sc_s=c_sf_s).
\end{eqnarray*}
Hence
\begin{eqnarray*}
&&\alpha^{(1,s)}_{(\vec{x}_{1},\mu_{1})}\lambda^{(s)}_{(\vec{x}_s,\mu_s)}(\vec{y}_s,\mu'_s)
=\lambda^{(s)}_{\alpha^{(1,s)}_{(\vec{x}_{1},\mu_{1})}(\vec{x}_s,\mu_s)}\alpha^{(1,s)}_{(\vec{x}_{1},\mu_{1})}(\vec{y}_s,\mu'_s),\end{eqnarray*}
and the result follows.
\end{proof}

\begin{lemma}\label{commute}
With the above notation,  we have that
$$\alpha^{(k,i)}_{(\vec{x}_k,\mu_k)}\circ\alpha^{(j,i)}_{(\vec{x}_j,\mu_j)}
=\alpha^{(j,i)}_{(\vec{x}_j,\mu_j)}\circ\alpha^{(k,i)}_{(\vec{x}_k,\mu_k)},$$
for
all  $i,j,k\in \{ 1,2,\dots ,s\}$, $j\neq i$, $k\neq i$ and $k\neq
j$.
\end{lemma}
\begin{proof}
This follows from the definition of the maps $\alpha^{(j,i)}$
(recall that $\alpha^{(j,i)}_{(\vec{x}_j,\mu_j)}=\id_{H_i}$ if
$(j,i)\notin\{ (l+1,l)\mid l=1,\dots ,s-1\}\cup\{ (s,1)\}$).
\end{proof}

By Lemmas~\ref{qsprimes}, \ref{auto} and~\ref{commute} the maps
$\alpha^{(j,i)}$ satisfy the hypothesis of
Proposition~\ref{morphism}. Therefore  the maps $\alpha^{(j,i)}$,
with $1\leq i ,j\leq s$, define an iterated matched product
$H_1\bowtie\dots\bowtie H_s$ of left ideals.

\begin{theorem}\label{sprimes}
With the above notation,  the left brace $H_1\bowtie\dots \bowtie
H_s$ is simple if and only if $c_i-\id$ is an automorphism for all
$1\leq i\leq s$.
\end{theorem}

\begin{proof} Let $I$
 be a nonzero ideal of $H_1\bowtie \dots\bowtie H_s$.
Let $(\vec{x}_1,\mu_1,\dots ,\vec{x}_{s},\mu_{s})\in I$ be a
nonzero element. Suppose that $\vec{x}_i\neq 0$ for some $i$.
Since the orders of the left braces $H_i$ are pairwise coprime, we
may assume that $\vec{x}_k=\vec{0}$ and $\mu_{k}=0$ for each $k\neq i$,
and $\vec{x}_i\neq \vec{0}$. Note that for
$(\vec{y}_j,Q_j(\vec{y}_j))\in H_j$ we have that
$q_j(\vec{y}_j,Q_j(\vec{y}_j))=0$, for every $j$. Hence, by (\ref{lambda})
\begin{eqnarray*}
&&\lambda^{(i)}_{(\vec{y}_i,Q_i(\vec{y}_i))}(\vec{x}_i,\mu_i)=(\vec{x}_i,\mu_i+b_i(\vec{y}_i,\vec{x}_i)).
\end{eqnarray*}
 So, $(\vec{0},0,\dots,\vec{0},b_i(\vec{y}_i,\vec{x}_i),\dots,
\vec{0},0)\in I$ (where $b_i(\vec{y}_i,\vec{x}_i)$ is in position
$2i$). Since $Q_i$ is non-degenerate and $\vec{x}_i\neq \vec{0}$,
there exists $\vec{y}_i$ such that $b_i(\vec{y}_i,\vec{x}_i)\neq 0$.
We may assume that $b_i(\vec{y}_i,\vec{x}_i)=p_i^{s_i}$, for some
$0\leq s_i<r_i$.

On the other hand, if $\vec{x}_k=0$ for all $k=1,\dots ,s$, then
$\mu_i\neq 0$ for some $i$. In this case, we also get that
$(\vec{0},0,\dots,\vec{0},p_i^{s_i},\dots, \vec{0},0)\in I$ (where
$p_i^{s_i}$ is in position $2i$),  for some $0\leq s_i<r_i$.

Thus, without loss of generality, we may assume that
$\vec{x}_l=\vec{0}$, for all $l=1,\dots ,s$,  and there exists $i$ such that $\mu_i=p_i^{s_i}$ and
$\mu_k=0$ for each $k\neq i$.  By
Lemma~\ref{ideal}, if $i=1$, then
\begin{eqnarray*}
\lefteqn{(\lambda_{(\vec{x}_1,\mu_1,\dots,
\vec{x}_{s},\mu_{s})}-\id)(\vec{y}_1,\mu'_1,\dots,
\vec{y}_{s},\mu'_{s})}\\
&=&((\lambda^{(1)}_{(\vec{x}_1,\mu_1)}-\id)(\vec{y}_1,\mu'_1),\vec{0},0,\dots,\vec{0},0,
(\alpha^{(1,s)}_{(\vec{x}_1,\mu_1)}-\id)(\vec{y}_{s},\mu'_{s}))\\
&=&((f_1^{\mu_1}-\id)(\vec{y}_1),0,\vec{0},0,\dots,\vec{0},0,
(c_s^{\mu_1}-\id)(\vec{y}_{s}),\\
&&\qquad v_s((\id+c_s+\dots+c_s^{\mu_1-1})(\vec{y}_s))^t) \in I,
\end{eqnarray*}
for all $(\vec{y}_1,\mu'_1,\dots, \vec{y}_{s},\mu'_{s})\in
H_1\bowtie \dots\bowtie H_s$.
If $1<i\leq s$, then
\begin{eqnarray*}
\lefteqn{(\lambda_{(\vec{x}_1,\mu_1,\dots,
\vec{x}_{s},\mu_{s})}-\id)(\vec{y}_1,\mu'_1,\dots,
\vec{y}_{s},\mu'_{s})}\\
&=&(\vec{0},0,\dots,(\alpha^{(i,i-1)}_{(\vec{x}_i,\mu_i)}-\id)(\vec{y}_{i-1},\mu'_{i-1}),(\lambda^{(i)}_{(\vec{x}_i,\mu_i)}-\id)(\vec{y}_i,\mu'_i),\dots,
\vec{0},0)\\
&=&(\vec{0},0,\dots,(c_{i-1}^{\mu_i}-\id)(\vec{y}_{i-1}),0,(f_i^{\mu_i}-\id)(\vec{y}_i),0,\dots,
\vec{0},0) \in I.
\end{eqnarray*}
Note that if $1<i\leq s$, then the endomorphism of
$(\mathbb{Z}/(p_{i-1}))^{n_{i-1}}$ induced by
$c_{i-1}^{p_i^{s_i}}-\id$ is nonzero and thus there exists
$\vec{y}_{i-1}$ such that
$(c_{i-1}^{p_i^{s_i}}-\id)(\vec{y}_{i-1})\notin
(p_{i-1}\mathbb{Z}/(p_{i-1}^{r_{i-1}}))^{n_{i-1}}$. If $i=1$, then
the endomorphism of $(\mathbb{Z}/(p_s))^{n_s}$ induced by
$c_s^{p_1^{s_1}}-\id$ is nonzero and thus there exists $\vec{y}_s$
such that $(c_{s}^{p_1^{s_1}}-\id)(\vec{y}_s)\notin
(p_s\mathbb{Z}/(p_s^{r_s}))^{n_s}$. Hence we may assume that
$$(\vec{0},0,\dots,\vec{z}_k,\nu_k,\dots, \vec{0},0)\in I,$$
for some $\vec{z}_k\in (\mathbb{Z}/(p_k^{r_k}))^{n_k}\setminus
(p_k\mathbb{Z}/(p_k^{r_k}))^{n_k}$, and some $\nu_k\in
\mathbb{Z}/(p_k^{r_k})$. As above, we get that
$$(\vec{0},0,\dots,\vec{0},b_k(\vec{y}_k,\vec{z}_k),\dots,
\vec{0},0)\in I$$ (where $b_k(\vec{y}_k,\vec{z}_k)$ is in position
$2k$). Since $Q_k$ is non-degenerate and $\vec{z}_k\in
(\mathbb{Z}/(p_k^{r_k}))^{n_k}\setminus
(p_k\mathbb{Z}/(p_k^{r_k}))^{n_k}$, there exists $\vec{y}_k\in
(\mathbb{Z}/(p_k^{r_k}))^{n_k}\setminus
(p_k\mathbb{Z}/(p_k^{r_k}))^{n_k}$ such that
$b_k(\vec{y}_k,\vec{z}_k)$ is an invertible element in
$\mathbb{Z}/(p_k^{r_k})$. Hence $$w_k=(\vec{0},0,\dots,
\vec{0},1,\dots,\vec{0},0)\in I$$ (where $1$ is in position $2k$).
We get by the above argument that
$$w_1=(\vec{0},1,\vec{0},0,\dots,\vec{0},0),\dots, w_{s}=(\vec{0},0,\dots,\vec{0},0,\vec{0},1)\in I.$$
Now it is easy to see that the ideal generated by $\{ w_1,\dots
,w_{s}\}$ is equal to
$$\{(\vec{z}_1,\nu_1,\dots,\vec{z}_{s},\nu_{s})\mid \nu_i\in \mathbb{Z}/(p_i^{r_i}),\;\vec{z}_i\in V_i\},$$
where
$$V_i=\langle(f_i^{a_1}c_{i}^{a_2}-f_i^{a'_1}c_{i}^{a'_2})(\vec{y})\mid
a_k,a'_k\in\mathbb{Z},\; \vec{y}\in
(\mathbb{Z}/(p_i^{r_i}))^{n_i}\rangle_+,$$ for $1\leq i\leq s$.
Note
that  $f_i$ and $c_i$ are elements of relative prime order in the
group $\Aut((\mathbb{Z}/(p_i^{r_i}))^{n_i})$. Hence the subgroup
generated by $f_i$  and  $c_{i}$ is $\langle f_ic_{i}\rangle$, for
$1\leq i\leq s$.  Therefore
$$V_i=\im(f_ic_{i}-\id),$$ for $1\leq i\leq s$.
Note that
\begin{eqnarray*}
(f_ic_{i}-\id)^{p_i^{r_i}}&=&f_i^{p_i^{r_i}}c_{i}^{p_i^{r_i}}-\id+p_ih_i\\
&=&c_{i}^{p_i^{r_i}}-\id+p_ih_i\\
&=&(c_{i}-\id)^{p_i^{r_i}}+p_ih'_i,
\end{eqnarray*}
for some $h_i,h'_i\in \End((\mathbb{Z}/(p_{i}^{r_i}))^{n_i})$.
Clearly,  $p_ih'_i$ is in the Jacobson radical
of the ring $\End((\mathbb{Z}/(p_{i}^{r_i}))^{n_i})$. Hence
$f_ic_{i}-\id$ is an automorphism if and only if $c_{i}-\id$ is an
automorphism. Thus the result follows.
\end{proof}

\section{Realisations of constructions of simple
braces}\label{examples}

In this section concrete examples of finite
simple left braces constructed as in Theorem~\ref{sprimes} are
given. Using Proposition~\ref{morphism} and Theorem~\ref{graph_simple}, we then also construct more examples of simple left braces.

First, we need some computations for matrices over $\Z$, which we
will later reduce to $\Z/(p^r)$. Consider the companion matrix of
the polynomial $x^{n-1}+x^{n-2} + \cdots + x+1$
$$
D=\begin{pmatrix}
0& 0& \cdots & 0&-1\\
1& 0& \cdots & 0& -1\\
0& \ddots & \ddots & \vdots &\vdots\\
\vdots &\ddots & \ddots & 0& -1\\
0& \cdots &  0 & 1 &-1
\end{pmatrix}\in GL_{n-1}(\Z),
$$
that has  multiplicative order $n$. Let $E\in M_{n-1}(\Z)$ be the
matrix given by
$$E=\frac{1}{2}(\Id+D^tD+(D^2)^tD^2 +\cdots + (D^{n-1})^t D^{n-1}).$$
By a straightforward computation, one can check that
$$E= \begin{pmatrix}
n-1& -1 &\cdots &-1 \\
-1 & \ddots &\ddots & \vdots\\
\vdots & \ddots& \ddots & -1 \\
-1 & \cdots & -1 & n-1
\end{pmatrix}.
$$
Since $D^{n}=\Id$,  it
follows that $D^tED=E$. It is an easy exercise to check that $\det
(E) =n^{n-2}$.

Consider the quadratic form $Q$ over $\Z$ defined as
$$Q(\vec{x})=\displaystyle\sum_{1\leq i<j\leq n-1} x_ix_j$$
for  $\vec{x}=(x_1,x_2,\dots,x_{n-1})\in \Z^{n-1}$. One can check
that
\begin{eqnarray}\label{whatweknowoverZ}
Q(D(\vec{x}))&=&Q(\vec{x})+\binom{n-1}{2} x_{n-1}^2-(n-1)\sum_{i=1}^{n-2} x_ix_{n-1}.
\end{eqnarray}

Let $s$ be an integer greater than $1$. Let $p_1,p_2,\dots ,p_{s}$
be different prime numbers and let $r_1,r_2,\dots ,r_{s}$ be
positive integers. Assume that $p_1,\dots ,p_{s-1}$ are odd. If
$p_s=2$, the we also assume that $r_s=1$.  Consider the following
matrices:
$$
D_i\equiv D\pmod{p_i^{r_i}}, \quad E_i\equiv E\pmod{p_i^{r_i}}$$
with $D_i,E_i\in
GL_{p_{i+1}^{r_{i+1}}-1}(\mathbb{Z}/(p_i^{r_i}))$, for $1\leq i<s$,
and $D_s,E_s\in GL_{p_{1}^{r_{1}}-1}(\mathbb{Z}/(p_s^{r_s}))$.
Recall that $\det (E_i) =n_{i}^{n_i -2}$, where
$n_i=p_{i+1}^{r_{i+1}}$, for $1\leq i<s$, and $n_s=p_{1}^{r_{1}}$.
The order of $D_j$ is $n_j$, and  $D_j^tE_jD_j=E_j$. Hence,
$D_j$ is an element of order $n_j$ in the orthogonal group
determined by the non-singular quadratic form corresponding to $E_j$
on the free module $(\mathbb Z/(p_j)^{r_j})^{n_j-1}$, if $p_j$ is
odd. Moreover $D_j-\Id$ is invertible  (because $1$ is not an eigenvalue of $D_{j}$ modulo $(p_{j})$).

If $p_s=2$, then we consider the quadratic form $Q_s$ on the vector
space $(\mathbb Z/(2))^{p_1^{r_1}-1}$ defined by $Q_s(x_1,\dots,
x_{p_1^{r_1}-1})=\sum_{1\leq i<j\leq p_1^{r_1}-1}x_ix_j$. In this
case, let $v_s=(0,\dots ,0,\binom{p_1^{r_1}-1}{2})\in (\mathbb
Z/(2))^{p_1^{r_1}-1}$. By (\ref{whatweknowoverZ}) we have that
\begin{eqnarray*}
&&Q_s(D_s(\vec{x}))=Q_s(\vec{x})+\binom{p_1^{r_1}-1}{2}
x_{p_1^{r_1}-1}=Q_s(\vec{x})+v_s\vec{x}^t.
\end{eqnarray*}

Let $0\leq r_i' \leq r_i$. Consider in $ M_{p_i^{r'_i}(n_i-1)}
(\mathbb Z/(p_i^{r_i}))$ the block diagonal matrices with
$p_{i}^{r'_i}$ blocks of degree $n_i -1$:
$$C_i= \begin{pmatrix}
D_i & 0 & \cdots &   0\\
0 & D_i & \ddots &  \vdots  \\
\vdots &  \ddots& \ddots & 0\\
0& \cdots & 0 & D_i \end{pmatrix} \quad {\rm and } \quad
 B_i= \begin{pmatrix}
E_i & 0 & \cdots  &  0\\
0 & E_i & \ddots  & \vdots  \\
\vdots & \ddots&\ddots & 0\\
0& \cdots &  0 & E_i \end{pmatrix}.$$

Consider the following block permutation matrix
$$F_i= \begin{pmatrix}
0& 0 & \cdots & 0 &  J_i\\
J_i & 0 & \ddots &  & 0  \\
0 & \ddots & \ddots  & \ddots & \vdots  \\
\vdots &   \ddots& J_i& 0 & 0\\
0& \cdots &  0 & J_i & 0\end{pmatrix},$$ where $J_i\in M_{n_i-1}(
\mathbb Z/(p_i^{r_i}))$ is the identity matrix. Notice that $F_i^t =
F_i^{-1}$ and
\begin{eqnarray*}
\lefteqn{ F_i^tB_iF_i}\\
& =&\begin{pmatrix}
0& J_i & \cdots & 0 &  0\\
0 & 0 & \ddots &  & 0  \\
0 & \ddots & \ddots  & \ddots & \vdots  \\
\vdots &   \ddots& 0& 0 & J_i\\
J_i& \cdots &  0 & 0 & 0\end{pmatrix}
\begin{pmatrix}
E_i & 0 & \cdots & \cdots &  0\\
0 & E_i & \ddots &  & \vdots  \\
\vdots & \ddots & \ddots  & \ddots & \vdots  \\
\vdots &  &  \ddots& E_i & 0\\
0& \cdots & \cdots& 0 & E_i \end{pmatrix}
\begin{pmatrix}
0& 0 & \cdots & 0 &  J_i\\
J_i & 0 & \ddots &  & 0  \\
0 & \ddots & \ddots  & \ddots & \vdots  \\
\vdots &   \ddots& J_i& 0 & 0\\
0& \cdots &  0 & J_i & 0\end{pmatrix} \\
&=&B_i.
\end{eqnarray*}
 Therefore, $F_i$ is an
element of  order $p_i^{r'_i}$ in the orthogonal  group determined
by the non-singular quadratic form corresponding to $B_i$ on the
free module $(\mathbb Z/(p_i^{r_i}))^{p_i^{r'_i}(n_i-1)}$,  if $p_i$
is odd.

For $p_s=2$, if $r'_s=0$, then $F_s$ is the identity matrix. If
$r'_s=1$, then we consider the quadratic form $Q'_s$ on the vector
space $(\mathbb Z/(2))^{2(n_s-1)}$ defined by $Q'_s(x_1,\dots,
x_{2(n_s-1)})=Q_s(x_1,\dots, x_{n_s-1})+Q_s(x_{n_s},\dots,
x_{2(n_s-1)})$ and the element $v'_s=(v_s,v_s)\in (\mathbb
Z/(2))^{2(n_s-1)}$. In this case $F_s$ is an element of order $2$ in
the orthogonal group determined by the non-singular quadratic form
$Q'_s$. We also have that
$Q'_s(C_s(\vec{x}))=Q'_s(\vec{x})+v'_s\vec{x}^t$.

Moreover, we have $F_i^{-1} C_iF_i=C_i$, so that $C_iF_i=F_iC_i$,
and $C_i-\Id$ is invertible because $D_i-\Id$ is invertible.

By Theorem~\ref{sprimes}, we can construct a simple left brace with
additive group
$(\mathbb{Z}/(p_1^{r_1}))^{p_1^{r'_1}(p_2^{r_2}-1)+1}\times
\cdots\times
(\mathbb{Z}/(p_{s-1}^{r_{s-1}}))^{p_{s-1}^{r'_{s-1}}(p_s^{r_s}-1)+1}$
$\times (\mathbb
Z/(p_{s}^{r_{s}}))^{p_{s}^{r'_{s}}(p_1^{r_1}-1)+1}.$  First
 take the quadratic form $Q_i$ corresponding to the matrix $B_i$
if $p_i$ is odd, and $f_i$ corresponding to the matrices $F_i$.
Further, take $c_{i}$ corresponding to the matrix $C_i$ for $1\leq
i\leq s$.

Note that simplicity follows from Theorem~\ref{sprimes} because
$c_i-\id$ is  invertible for $i=1,2,\dots ,s$.

One can also construct concrete examples of simple left braces using
Theorem~\ref{graph_simple}.
For this we will need the following
lemma.

\begin{lemma}\label{qi}
With the notation of Section~\ref{simple}, consider the left brace
$H_1\bowtie\dots \bowtie H_s$ of Theorem~\ref{sprimes}. Then the map
$$\begin{array}{lccc}
\varphi_i\colon &(H_1\bowtie\dots \bowtie H_s,\cdot)&\longrightarrow
&\mathbb{Z}/(p_i^{r_i})\\
&(\vec{x}_1,\mu_1,\dots ,\vec{x}_s,\mu_s)&\mapsto
&q_i(\vec{x}_i,\mu_i)\end{array}$$ is a homomorphism of groups.
\end{lemma}

\begin{proof}
Let $(\vec{x}_1,\mu_1,\dots ,\vec{x}_s,\mu_s),
(\vec{y}_1,\mu'_1,\dots ,\vec{y}_s,\mu'_s)\in H_1\bowtie \dots
\bowtie H_s$. Using the formula (\ref{form_lambda}) of the lambda map, we get
\begin{eqnarray*}
\lefteqn{\varphi_i((\vec{x}_1,\mu_1,\dots ,\vec{x}_s,\mu_s)\cdot
(\vec{y}_1,\mu'_1,\dots ,\vec{y}_s,\mu'_s))}\\
&=&\varphi_i((\vec{x}_1,\mu_1,\dots
,\vec{x}_s,\mu_s)+\lambda_{(\vec{x}_1,\mu_1,\dots
,\vec{x}_s,\mu_s)}(\vec{y}_1,\mu'_1,\dots ,\vec{y}_s,\mu'_s))\\
&=&\varphi_i((\vec{x}_1,\mu_1,\dots
,\vec{x}_s,\mu_s)+(\lambda^{(1)}_{(\vec{x}_1,\mu_1)}\alpha^{(2,1)}_{(\vec{x}_2,\mu_2)}(\vec{y}_1,\mu'_1),\dots
,\\
&&\qquad
\lambda^{(s-1)}_{(\vec{x}_{s-1},\mu_{s-1})}\alpha^{(s,s-1)}_{(\vec{x}_s,\mu_s)}(\vec{y}_{s-1},\mu'_{s-1}),
\lambda^{(s)}_{(\vec{x}_s,\mu_s)}\alpha^{(1,s)}_{(\vec{x}_1,\mu_1)}
(\vec{y}_s,\mu'_s)))\\
&=&\left\{\begin{array}{ll}
q_i((\vec{x}_i,\mu_i)+(\lambda^{(i)}_{(\vec{x}_i,\mu_i)}\alpha^{(i+1,i)}_{(\vec{x}_{i+1},\mu_{i+1})}(\vec{y}_i,\mu'_i)))\quad \mbox{if }1\leq i<s\\
q_s((\vec{x}_s,\mu_s)+(\lambda^{(s)}_{(\vec{x}_s,\mu_s)}\alpha^{(1,s)}_{(\vec{x}_{1},\mu_{1})}(\vec{y}_s,\mu'_s)))\quad \mbox{if } i=s\\
\end{array}\right.\\
\end{eqnarray*}
Now we have that
\begin{eqnarray*}
\lefteqn{q_i((\vec{x}_i,\mu_i)+(\lambda^{(i)}_{(\vec{x}_i,\mu_i)}\alpha^{(i+1,i)}_{(\vec{x}_{i+1},\mu_{i+1})}(\vec{y}_i,\mu'_i)))}\\
&=&q_i(\vec{x}_i,\mu_i)+q_i(\alpha^{(i+1,i)}_{(\vec{x}_{i+1},\mu_{i+1})}(\vec{y}_i,\mu'_i))\quad\mbox{(by
(\ref{q}))}\\
&=&q_i(\vec{x}_i,\mu_i)+q_i(\vec{y}_i,\mu'_i)\quad\mbox{(by
Lemma~\ref{q2primes})}
\end{eqnarray*}
and
\begin{eqnarray*}
\lefteqn{q_s((\vec{x}_s,\mu_s)+(\lambda^{(s)}_{(\vec{x}_s,\mu_s)}\alpha^{(1,s)}_{(\vec{x}_{1},\mu_{1})}(\vec{y}_s,\mu'_s)))}\\
&=&q_s(\vec{x}_s,\mu_s)+q_s(\alpha^{(1,s)}_{(\vec{x}_{1},\mu_{1})}(\vec{y}_s,\mu'_s))\quad\mbox{(by
(\ref{q}))}\\
&=&q_s(\vec{x}_s,\mu_s)+q_s(\vec{y}_s,\mu'_s)\quad\mbox{(by
Lemma~\ref{q2primes})}
\end{eqnarray*}
Therefore, the result follows.
\end{proof}

\begin{lemma}\label{brauto}
With the notation of Section~\ref{simple}, assume that for some
$i\in\{ 1,\dots, s\}$ there exists a divisor $m$ of $n_i$ such that
$$
Q_i(\vec{x}_i)=\sum_{j=1}^{m}Q(\vec{x}_{i,j}),
$$
$$
f_i(\vec{x}_i)=(f(\vec{x}_{i,1}),\dots ,f(\vec{x}_{i,m})), \quad {\rm and }  \quad
c_{i}(\vec{x}_i)=(c(\vec{x}_{i,1}),\dots ,c(\vec{x}_{i,m})),
$$
where $\vec{x}_i=(\vec{x}_{i,1},\dots ,\vec{x}_{i,m})$,
$\vec{x}_{i,j}\in (\mathbb{Z}/(p_i^{r_i}))^{\frac{n_i}{m}}$, $Q$ is
a nonsingular quadratic form over
$(\mathbb{Z}/(p_i^{r_i}))^{\frac{n_1}{m}}$, $f$ is an element
in the orthogonal group determined by $Q$, if $i\neq s$, then $c$
also is an element in the orthogonal group determined by $Q$,
and, if $i=s$, then $c$ is an automorphism of
$(\mathbb{Z}/(p_s^{r_s}))^{\frac{n_s}{m}}$, $v_s=(v,\dots ,v)$, for
some $v\in (\mathbb{Z}/(p_s^{r_s}))^{\frac{n_s}{m}}$, and
$Q(c(\vec{x}))=Q(\vec{x})+v\vec{x}^t$. Let $\sigma\in \sym_{m}$.
Then the map
$$\psi_{\sigma}\colon H_1\bowtie\dots \bowtie H_s\longrightarrow
H_1\bowtie \dots \bowtie  H_s,$$ defined by
\begin{eqnarray*}
\lefteqn{\psi_{\sigma}(\vec{x}_1,\mu_1,\dots
,\vec{x}_s,\mu_s)}\\
 &=&(\vec{x}_1,\mu_1,\dots
,\vec{x}_{i-1},\mu_{i-1},\vec{x}_{i,\sigma(1)},\dots,
\vec{x}_{i,\sigma(m)},\mu_i,\vec{x}_{i+1},\mu_{i+1},\dots
,\vec{x}_s,\mu_s),
\end{eqnarray*}
is an
automorphism of the left brace $H_1\bowtie\dots \bowtie  H_s$ of
Theorem~\ref{sprimes}.
\end{lemma}

\begin{proof}
Let $(\vec{x}_1,\mu_1,\dots ,\vec{x}_s,\mu_s),
(\vec{y}_1,\mu'_1,\dots ,\vec{y}_s,\mu'_s)\in H_1\bowtie \dots
\bowtie H_s$. Clearly, $\psi_{\sigma}$ is an automorphism of the
additive group of the left brace $H_1\bowtie \dots \bowtie H_s$.
Thus, to prove the result, it is enough to show that
\begin{eqnarray}\label{psilambda}
\lefteqn{\psi_{\sigma}(\lambda_{(\vec{x}_1,\mu_1,\dots
,\vec{x}_s,\mu_s)}(\vec{y}_1,\mu'_1,\dots
,\vec{y}_s,\mu'_s))}\notag\\
&=&\lambda_{\psi_{\sigma}(\vec{x}_1,\mu_1,\dots
,\vec{x}_s,\mu_s)}\psi_{\sigma}(\vec{y}_1,\mu'_1,\dots
,\vec{y}_s,\mu'_s)
\end{eqnarray}
Note that, if $1\leq i<s$, then the component $i$ of
$\lambda_{(\vec{x}_1,\mu_1,\dots
,\vec{x}_s,\mu_s)}(\vec{y}_1,\mu'_1,\dots ,\vec{y}_s,\mu'_s)$ is
$\lambda^{(i)}_{(\vec{x}_i,\mu_i)}\alpha^{(i+1,i)}_{\vec{x}_{i+1},\mu_{i+1}}(\vec{y}_i,\mu'_i)$.

The component $s$ of $\lambda_{(\vec{x}_1,\mu_1,\dots
,\vec{x}_s,\mu_s)}(\vec{y}_1,\mu'_1,\dots ,\vec{y}_s,\mu'_s)$ is
$\lambda^{(s)}_{(\vec{x}_s,\mu_s)}\alpha^{(1,s)}_{\vec{x}_{1},\mu_{1}}(\vec{y}_s,\mu'_s)$.

Since $q_i(\vec{x}_i,\mu_i)=\mu_i-\sum_{j=1}^mQ(\vec{x}_{i,j})$, and
$b_i(\vec{x}_i,\vec{z}_i)=\sum_{j=1}^mb(\vec{x}_{i,j},\vec{z}_{i,j})$,
where
$b(\vec{x}_{i,j},\vec{z}_{i,j})=Q(\vec{x}_{i,j}+\vec{z}_{i,j})-Q(\vec{x}_{i,j})-Q(\vec{z}_{i,j})$,
the reader  easily can  check (\ref{psilambda}) using the form of
$f_i$, $c_{i}$ (and $v_s$ in the case where $i=s$).
\end{proof}

Because of Proposition~\ref{morphism} and Theorem~\ref{graph_simple}, we are now in a position to construct more concrete examples of  simple left
braces that are iterated matched products of left ideals.

\begin{example}\label{ex}{\rm
Let $p_1,p_2,p_3,p_4$ be different prime numbers. For simplicity,
assume that  all are odd. We can construct as above two
simple left braces $H_1\bowtie H_2$ and $H_3\bowtie H_4$, where
$$H_1=H(p_1,p_4(p_2-1),Q_1,\id),\; H_2=H(p_2,p_1-1,Q_2,\id),$$
$$H_3=H(p_3,p_2(p_4-1),Q_3,\id),\; H_4=H(p_4,p_3-1,Q_4,\id),$$ for some nonsingular quadratic forms $Q_j$ where
\begin{eqnarray*}
&&Q_1(x_1,\dots,x_{p_4(p_2-1)})=\sum_{j=0}^{p_4-1}Q'_1(x_{1+(p_2-1)j},\dots ,x_{p_2-1+(p_2-1)j}),\\
&&Q_3(y_1,\dots
,y_{p_2(p_4-1)})=\sum_{k=0}^{p_2-1}Q'_3(y_{1+(p_4-1)k},\dots
,y_{p_4-1+(p_4-1)k}),
\end{eqnarray*}
$Q'_1$ is a nonsingular quadratic form over
$(\mathbb{Z}/(p_1))^{p_2-1}$, $Q_2$ is a nonsingular quadratic form
over $(\mathbb{Z}/(p_2))^{p_1-1}$, $Q'_3$ is a nonsingular quadratic
form over $(\mathbb{Z}/(p_3))^{p_4-1}$ and $Q_4$ is a nonsingular
quadratic form over $(\mathbb{Z}/(p_4))^{p_3-1}$. Let $c'_{1}$ be an
element of order $p_2$ in the orthogonal group determined by $Q'_1$,
let $d_{1}$ be an element of order $p_1$  in the orthogonal group
determined by $Q_2$, let $c'_{2}$ be an element of order $p_4$  in
the orthogonal group determined by $Q'_3$, let $d_{2}$ be an element
of order $p_3$  in the orthogonal group determined by $Q_4$. Assume
that the endomorphisms $c'_1-\id$, $d_1-\id$, $c'_2-\id$ and
$d_2-\id$ are invertible.  Note that the maps $c_1\in
\Aut((\mathbb{Z}/(p_1))^{p_4(p_2-1)})$ and $c_2\in
\Aut((\mathbb{Z}/(p_3))^{p_2(p_4-1)})$ defined by
$$c_1(\vec{x}_1,\dots ,\vec{x}_{p_4})=(c'_1(\vec{x}_1),\dots ,c'_1(\vec{x}_{p_4}))$$
and
$$c_2(\vec{y}_1,\dots ,\vec{y}_{p_2})=(c'_2(\vec{y}_1),\dots ,c'_2(\vec{y}_{p_4})),$$
for $\vec{x}_j\in (\mathbb{Z}/(p_1))^{p_2-1}$ and $\vec{y}_k\in
(\mathbb{Z}/(p_3))^{p_4-1}$, are elements  in the orthogonal
groups determined by $Q_1$ and $Q_3$ respectively. The actions of
the matched pairs $(H_1,H_2,\alpha^{(2,1)},\alpha^{(1,2)})$ and
$(H_3,H_4,\alpha^{(4,3)},\alpha^{(3,4)})$ are defined by
$\alpha^{(j,i)}(\vec{z},\mu')=\alpha^{(j,i)}_{(\vec{z},\mu')}$ and
\begin{eqnarray*}\alpha^{(2,1)}_{(\vec{z},\mu')}(\vec{x},\mu)&=&(c_1^{q_2(\vec{z},\mu')}(\vec{x}),\mu),\\
\alpha^{(1,2)}_{(\vec{x},\mu)}(\vec{z},\mu')&=&(d_1^{q_1(\vec{x},\mu)}(\vec{z}),\mu'),\\
\alpha^{(4,3)}_{(\vec{u},\nu')}(\vec{y},\nu)&=&(c_2^{q_4(\vec{u},\nu')}(\vec{y}),\nu),\\
\alpha^{(3,4)}_{(\vec{y},\nu)}(\vec{u},\nu')&=&(d_2^{q_3(\vec{y},\nu)}(\vec{u}),\nu'),
\end{eqnarray*}
Let $\sigma_1$ and $\sigma_2$ be the cyclic permutations
$\sigma_1=(1,2,\dots, p_2)$ and $\sigma_2=(1,2,\dots, p_4)$. By
Lemma~\ref{brauto}, the maps $\psi_{\sigma_1}\colon H_3\bowtie
H_4\longrightarrow H_3\bowtie H_4$, defined by
$$\psi_{\sigma_1}(\vec{y}_1,\dots,\vec{y}_{p_2},\nu,\vec{u},\nu')=(\vec{y}_{\sigma_1(1)},\dots,\vec{y}_{\sigma_1(p_2)},\nu,\vec{u},\nu'),$$
and $\psi_{\sigma_2}\colon H_1\bowtie H_2\longrightarrow H_1\bowtie
H_2$, defined by
$$\psi_{\sigma_2}(\vec{x}_1,\dots,\vec{x}_{p_4},\mu,\vec{z},\mu')=(\vec{x}_{\sigma_2(1)},\dots,\vec{x}_{\sigma_2(p_4)},\mu,\vec{z},\mu'),$$
are automorphisms of left braces. We define $\alpha\colon
(H_3\bowtie H_4, \cdot )\longrightarrow \Aut(H_1\bowtie
H_2,+,\cdot)$ and $\beta\colon (H_1\bowtie H_2,\cdot
)\longrightarrow \Aut(H_3\bowtie H_4,+,\cdot)$ by
$$\alpha(\vec{y}_1,\nu,\vec{u},\nu')=\psi_{\sigma_2}^{q_4(\vec{u},\nu')} \quad
{\rm and } \quad
\beta(\vec{x}_1,\mu,\vec{z},\mu')=\psi_{\sigma_1}^{q_2(\vec{z},\mu')}.$$
By Lemma~\ref{qi}, $\alpha$ and $\beta$ are homomorphisms of groups.
Denote $\alpha(\vec{y}_1,\nu,\vec{u},\nu')$ by
$\alpha_{(\vec{y}_1,\nu,\vec{u},\nu')}$ and denote
$\beta(\vec{x}_1,\mu,\vec{z},\mu')$ by
$\beta_{(\vec{x}_1,\mu,\vec{z},\mu')}$. Note that
$$
\alpha_{\beta_{(\vec{x},\mu,\vec{z},\mu')}(\vec{y},\nu,\vec{u},\nu')}\; =\; \alpha_{\psi_{\sigma_1}^{q_2(\vec{z},\mu')}(\vec{y},\nu,\vec{u},\nu')}
\; =\; \psi_{\sigma_2}^{q_4(\vec{u},\nu')}
\; = \; \alpha_{(\vec{y},\nu,\vec{u},\nu')}
$$
 and
$$\beta_{\alpha_{(\vec{y},\nu,\vec{u},\nu')}(\vec{x},\mu,\vec{z},\mu')}\; =\; \beta_{\psi_{\sigma_2}^{q_4(\vec{u},\nu')}(\vec{x},\mu,\vec{z},\mu')}
\; =\; \psi_{\sigma_1}^{q_2(\vec{z},\mu')}
\; = \; \beta_{(\vec{x},\mu,\vec{z},\mu')}.
$$
Hence, by
Proposition~\ref{morphism}, $(H_1\bowtie H_2, H_3\bowtie H_4,\alpha,
\beta)$ is a matched pair of left braces and the matched product
$(H_1\bowtie H_2)\bowtie (H_3\bowtie H_4)$ is a simple left brace.
}\end{example}

\begin{remark}{\rm
Clearly Example~\ref{ex} can be generalized to matched products of
arbitrary two simple braces of coprime orders constructed as in
Theorem~\ref{sprimes};  namely by
 replacing $H_1 \bowtie H_2$ and $H_3 \bowtie H_4$ by any braces as in Theorem~\ref{sprimes}.
 Even more, we can construct simple braces as iterated matched product of left ideals $B_1\bowtie \dots\bowtie B_n$, where every $B_i$ is a simple left brace as in Theorem~\ref{sprimes}.}
\end{remark}

\section{Comments and questions}\label{comments}

In view  of Remark~\ref{Sylow_remark} and the comment following Theorem~\ref{graph_simple}, the following
 seems to be a crucial step in the general program of describing all finite simple left braces.

\begin{problem}  Describe the structure of all left braces of order $p^n$, for a prime $p$.
And describe the group $\Aut (B,+,\cdot)$ of automorphisms for all
such left braces.
\end{problem}

Given two distinct primes $p$ and $q$, in the previous sections, we
have constructed many simple left braces of order
$p^{\alpha}q^{\beta}$ with some  natural restrictions on the
positive integers $\alpha$ and $\beta$.
As mentioned in the introduction,  these natural restrictions come from a recent result of
Smoktunowicz, that yields a necessary condition for a left brace of order
$p^{\alpha} q^{\beta}$ to be simple. Namely,  $q\mid (p^i-1)$ and
$p\mid (q^j-1)$, for some $1\leq i\leq \alpha$  and $1\leq j\leq
\beta$. Hence, the following seems
to be an interesting question.

\begin{problem}
Determine for which prime numbers $p$, $q$ and positive integers
$\alpha, \beta$, there exists a simple left brace of cardinality
$p^\alpha q^\beta$.
\end{problem}

An easy observation shows that not all such orders can occur.

\begin{remark}\label{answerAgata}
{\rm Let $G$ be a group of order $p^nq$, where $n$ is the
multiplicative order of $p$ in $(\Z/(q))^\star$, and $p$ and $q$ are
distinct prime numbers.   Because of the Sylow theorems, it is easy
to see  that either a Sylow $p$-subgroup or a Sylow $q$-subgroup of
$G$ is a normal subgroup. By Proposition 6.1 in \cite{B3}, every
normal Sylow subgroup of the multiplicative group of a left brace
$B$ is an ideal of $B$. Therefore,  $G$ does not admit a structure
of a simple left brace.
For instance, any group of order $2^3\cdot 7$ has a normal
Sylow $2$-subgroup or a normal Sylow $7$-subgroup, so there is no simple left brace of order $2^3\cdot 7$.}
\end{remark}

\vspace{30pt}
 \noindent \begin{tabular}{llllllll}
 D. Bachiller && F. Ced\'o  \\
 Departament de Matem\`atiques &&  Departament de Matem\`atiques \\
 Universitat Aut\`onoma de Barcelona &&  Universitat Aut\`onoma de Barcelona  \\
08193 Bellaterra (Barcelona), Spain    && 08193 Bellaterra (Barcelona), Spain \\
 dbachiller@mat.uab.cat &&  cedo@mat.uab.cat\\
   &&   \\
E. Jespers && J. Okni\'{n}ski  \\ Department of Mathematics &&
Institute of Mathematics
\\  Vrije Universiteit Brussel && Warsaw University \\
Pleinlaan 2, 1050 Brussel, Belgium && Banacha 2, 02-097 Warsaw, Poland\\
efjesper@vub.ac.be&& okninski@mimuw.edu.pl
\end{tabular}


\begin{thebibliography}{99}
\itemsep=-2pt
\bibitem{B} D. Bachiller, Classification of braces of order $p^3$, J.
Pure Appl. Algebra 219 (2015), 3568--3603.
\bibitem{B2} D. Bachiller, Counterexample to a conjecture about
braces, J. Algebra 453 (2016), 160--176.
\bibitem{B3} D. Bachiller, Extensions, matched products and simple braces,
arXiv:1511.08477v3[math.GR].
\bibitem{BCJ} D. Bachiller, F. Ced\'o and E. Jespers, Solutions of
the Yang-Baxter equation associated with a left brace,  J. Algebra 463 (2016), 80--102.
\bibitem{BenDavid} N. Ben David, On groups of central type
and involutive Yang-Baxter groups: a cohomological approach, Ph.D.
thesis, The Technion-Israel Institute of Technology, Haifa, 2012.
\bibitem{BDG2} N. Ben David and Y. Ginosar, On groups of $I$-type and
involutive Yang-Baxter groups, J. Algebra 458 (2016), 197--206.
\bibitem{brown} K. A. Brown and K. R. Goodearl, Lectures on algebraic
quantum groups, Birh\"auser Verlag, Basel, 2002.
\bibitem{CCS2015} F. Catino, I. Colazzo and P. Stefanelli, On regular
subgroups of the affine group,
Bull. Aust. Math. Soc.  91 (2015),  76--85.
\bibitem{CCS} F. Catino, I. Colazzo and P. Stefanelli, Regular
subgroups of the afine group and asymmetric product of braces,  J.
Algebra 455 (2016), 164--182.
\bibitem{CR} F. Catino and R. Rizzo, Regular subgroups of the affine
group and radical circle algebras, Bull. Aust. Math. Soc. 79 (2009),
no. 1, 103--107.
\bibitem{CGIS} F. Ced"o, T. Gateva-Ivanova and A. Smoktunowicz,
On  the Yang-Baxter equation and left nilpotent left braces, J. Pure
Appl. Algebra (2016),\\
 http://dx.doi.org/10.1016/j.jpaa.2016.07.014
\bibitem{CJO} F. Ced\'o, E. Jespers and J. Okni\'{n}ski,
Retractability  of the set theoretic solutions of the Yang-Baxter
equation, Adv. Math. 224 (2010), 2472--2484.
\bibitem{CJO2} F. Ced\'o, E. Jespers
and J. Okni\'{n}ski, Braces and the
 Yang-Baxter equation, Commun. Math. Phys. 327 (2014), 101--116.
\bibitem{CJR} F. Ced\'o, E. Jespers and \'{A}. del R\'{\i}o, Involutive
Yang-Baxter groups, Trans. Amer. Math. Soc. 362 (2010), 2541--2558.
\bibitem{Cohn} P. M. Cohn, Universal Algebra, Encyclopedia Math. App., vol. 6, Springer, 1981.
\bibitem{drinfeld} V. G. Drinfeld, On some unsolved problems in quantum group theory.
Quantum Groups, Lecture Notes Math. 1510, Springer-Verlag, Berlin,
1992, 1--8.
\bibitem{ESS} P. Etingof, T. Schedler and A.  Soloviev,  Set-theoretical solutions
 to the quantum Yang-Baxter equation, Duke Math. J. 100 (1999), 169--209.
\bibitem{GI} T. Gateva-Ivanova, A combinatorial approach to the
set-theoretic solutions of the Yang-Baxter equation, J. Math. Phys.
45 (2004), 3828--3858.
\bibitem{GIC} T. Gateva-Ivanova and P. Cameron, Multipermutation
solutions of the Yang-Baxter equation, Comm. Math. Phys. 309 (2012),
583--621.
\bibitem{gat-maj} T. Gateva-Ivanova and S. Majid,
Matched pairs approach to set theoretic solutions of the Yang-Baxter
equation, J. Algebra 319 (2008), no. 4, 1462--1529.
\bibitem{GIVdB} T. Gateva-Ivanova and M. Van den Bergh,
  Semigroups of $I$-type, J. Algebra 206 (1998), 97--112.
\bibitem{GI-braces} T. Gateva-Ivanova, Set-theoretic solutions of the
Yang-Baxter equation, braces and symmetric groups, preprint arXiv:
1507.02602v2 [math.QA].
\bibitem{Ven1} L. Guarnieri  and L. Vendramin, Skew braces and the Yang-Baxter equation,
Mathematics of Computation, to appear,
arXiv: 1151.03171v3 [mathQA].
\bibitem{H} P. Heged\H{u}s, Regular subgroups of the affine group,
J. Algebra  225 (2000), 740--742.
\bibitem{JO} E. Jespers and J. Okni\'{n}ski, Monoids and groups of
  $I$-type, Algebr. Represent. Theory 8 (2005), 709--729.
\bibitem{JObook} E. Jespers and J. Okni\'{n}ski,  Noetherian
Semigroup Algebras,
 Springer, Dordrecht 2007.
\bibitem{K} C. Kassel, Quantum Groups, Springer-Verlag, 1995.
\bibitem{Vend-Lebed} V. Lebed and L. Vendramin, Cohomology and
extensions of braces, Pacific J. Math. 284 (2016), 191--212.
\bibitem{Rump1} W. Rump, A decomposition theorem for square-free
unitary solutions of the quantum Yang-Baxter equation, Adv. Math.
193 (2005), 40--55.
\bibitem{Rump3} W. Rump, Modules over braces, Algebra Discrete Math. (2006), 127--137.
\bibitem{Rump} W. Rump, Braces, radical rings, and the quantum Yang-Baxter
 equation, J. Algebra 307 (2007), 153--170.
\bibitem{Rump2} W. Rump, Classification of cyclic braces, J. Pure Appl. Algebra 209 (2007),
671--685.
\bibitem{Rump4} W. Rump, Generalized radical rings, unknotted biquandles, and quantum
groups, Colloq. Math. 109 (2007), 85--100.
\bibitem{Rump5} W. Rump, Semidirect products in algebraic logic and solutions of the
quantum Yang-Baxter equation, J. Algebra Appl. 7 (2008), 471--490.
\bibitem{Rump6} W. Rump, Addendum to ``Generalized radical rings, unknotted biquandles,
and quantum groups'' (Colloq. Math. 109 (2007), 85--100),  Colloq.
Math. 117 (2009),  295--298.
\bibitem{Rump7} W. Rump, The brace of a classical group, Note Math.
34 (2014), 115--144.
\bibitem{Smokt} A. Smoktunowicz, A note on set-theoretic solutions
of the Yang-Baxter equation, J. Algebra (2016),\\
http://dx.doi.org/10.1016/j.jalgebra.2016.04.015
\bibitem{Tak} M. Takeuchi, Survey on matched pairs of groups. An
elementary approach to the ESS-LYZ theory, Banach Center Publ. 61
(2003), 305--331.
\bibitem{ven2016}  L. Vendramin,  Extensions of set-theoretic solutions of the Yang-Baxter equation and a conjecture of Gateva-Ivanova, J. Pure Appl. Algebra 220 (2016),  2064--2076.
\bibitem{Yang} C. N. Yang, Some exact results for the many-body
problem in one dimension with repulsive delta-function interaction,
Phys. Rev. Lett. 19 (1967), 1312--1315.
\end{thebibliography}
\end{document}